\documentclass[11pt]{amsart}
\usepackage{amsmath}
\usepackage{amsthm}
\usepackage{amsfonts}
\usepackage{amssymb}
\usepackage{amscd}
\numberwithin{equation}{section}
\newtheorem{theorem}{Theorem}[section]
\newtheorem{cor}[theorem]{Corollary}

\newtheorem{lemma}[theorem]{Lemma}
\newtheorem{prop}[theorem]{Proposition}
\theoremstyle{definition}
\newtheorem{defi}{Definition}[section]
\newtheorem{rem}[defi]{Remark}

\newcommand\half{\frac{1}{2}}

\newcommand\ov{\overline}

\newcommand\be{\beta}

\newcommand\g{\mathfrak g}
\newcommand\ga{\widehat{\mathfrak g}}
\newcommand\h{\mathfrak h}

\newcommand\D{\Delta}
\renewcommand\l{\lambda}
\newcommand\Dp{\Delta^+}

\renewcommand\d{\delta}

\renewcommand\a{\alpha}
\renewcommand\aa{\mathfrak a}

\renewcommand\k{\mathfrak k}

\newcommand{\Z}{\mathbb Z}
\newcommand\nat{\mathbb N}
\newcommand\ganz{\mathbb Z}

\newcommand\s{\sigma}
\renewcommand\L{\Lambda}

\renewcommand\aa{\mathfrak a}

\newcommand\e{\epsilon}
\newcommand\C{\mathbb C}

\newcommand\p{\mathfrak p}

\newcommand{\vac}{{\bf 1}}
\newcommand{\bea}{\begin{eqnarray}}
\newcommand{\eea}{\end{eqnarray}}
\begin{document}
\title[Finite vs infinite decompositions in  conformal embeddings]{Finite vs infinite decompositions in  conformal embeddings}
\author[Adamovi\'c, Kac, M\"oseneder, Papi, Per\v{s}e]{Dra{\v z}en~Adamovi\'c}\thanks{D.A.  and O.P. are partially supported by the Croatian Science Foundation under the project 2634 and by
the Croatian Scientific Centre of Excellence QuantixLie..}
\author[]{Victor~G. Kac}
\author[]{Pierluigi~M\"oseneder Frajria}
\author[]{Paolo~Papi}
\author[]{Ozren~Per\v{s}e}
\begin{abstract} Building on work of the first and last author, we prove that an embedding of simple affine vertex algebras $V_{\mathbf{k}}(\g^0)\subset V_{k}(\g)$, corresponding to an embedding of a maximal equal rank reductive subalgebra $\g^0$ into a simple Lie algebra $\g$, 
  is conformal if and only if the corresponding central charges are equal. We classify the equal rank  conformal embeddings. Furthermore we describe, in almost all cases,  when $V_{k}(\g)$ decomposes finitely as a $V_{\mathbf{k}}(\g^0)$-module.
\end{abstract}
\maketitle
\section{Introduction}
Let $V$ and $W$ be vertex algebras equipped with Virasoro elements $\omega_V$, $\omega_W$ and assume that $W$ is a vertex subalgebra of $V$. 
\begin{defi}\label{1}We say that $W$ is conformally embedded in $V$ if $\omega_V=\omega_W$.\end{defi}
\par
In this paper we deal with following problems.
\begin{enumerate}
\item Classify conformal embeddings when $V, W$ are  affine vertex algebras endowed with  $\omega_V$, $\omega_W$ given by Sugawara construction.
\item Decide whether the decomposition of $V$ as a $W$-module is finite, and in such a  case find the explicit decomposition.
\end{enumerate}
The general definition of conformal embedding introduced above is a natural generalization of the following notion, which has been popular in physics literature in the mid 80's, due to its  relevance for string compactifications.\par
Let $\g$ be a semisimple finite-dimensional complex Lie algebra and $\g^0$  a reductive subalgebra of $\g$. The embedding $\g^0\hookrightarrow \g$ is  called conformal    if the central charge of the 
Sugawara construction for
the affinization $\ga$,  acting on a level $1$ integrable module, equals that for the natural embedding  of $\widehat{\g^0}$ in $\ga$. Such an embedding is called maximal if no 
reductive subalgebra $\aa$ with $\g^0\subsetneq \aa\subsetneq\g$ embeds conformally in $\g$.\par
Maximal conformal embeddings were classified
in \cite{SW},  \cite{AGO}, and the corresponding decompositions are described in  \cite{KWW}, \cite{KS}, \cite{CKMP}. In the  vertex algebra framework the definition can be rephrased as follows: the simple affine vertex  algebras $V_{\bf k}(\g^0)$ and $V_1(\g)$ have the same
Sugawara conformal vector for some multiindex $\bf k$ of levels. 
One may wonder whether  the embedding $V_{\bf k}(\g^0)\subset V_k(\g)$ is conformal according to Definition \ref{1} for some  $k$, not necessarily $1$.\par
Section \ref{tre} is devoted  to answer question (1). We prove    that equality of central charges still detects conformality of maximal equal rank subalgebras: see Theorem \ref{main}.  In Proposition \ref{critnonmax}, we deal with non maximal equal rank embeddings. The proof of Theorem \ref{main} is elementary: it is obtained by combining  the results of \cite{A} with general results of Panyushev \cite{PLMS}Ê on combinatorics of root systems. \par
Question (2) has appeared many times in literature. We give a complete answer in  Section \ref{FD} when  $\g^0$ is semisimple and an almost complete answer in Section \ref{FDHS} when $\g^0$ has a nonzero center. The missing cases are listed in Remark \ref{missing}. Explicit decompositions are listed in Section \ref{6}.\par
To prove that  a conformal pair has finite decomposition, we use an enhancement, given in Theorems  \ref{general}  and \ref{generalhs} of results from \cite{A}. Infinite decompositions are settled by exhibiting 
infinitely many singular vectors: see Proposition \ref{43} and Corollary \ref{46} for the semisimple case.
 As a byproduct of our analysis   we obtain, in the semisimple case,  a criterion for infinite decomposition in terms of the existence of a $\widehat{\g^0}$-singular vector having  conformal weight  $2$: see Proposition \ref{cw}. 
 Other cases with infinite decomposition when $\g^0$ has  a nonzero center are dealt with in Theorem \ref{infinidec}.  The methods to prove this theorem are a combination of the ideas  used in the semisimple case with explicit realization of $V_k(\g)$, more precisely,  the Kac-Wakimoto free field realization of $V_{-1}(sl(n+1))$ \cite{Kw}Ê and Adamovi\'c's recent realization of $V_{-3/2}(sl(3))$ in the tensor product of the $N = 4$ superconformal vertex algebra with a suitable lattice vertex algebra \cite{A-2014}.
 \par 
It is worthwhile to note that  the  methods used in the proof of Theorems \ref{51} and  \ref{infinidec}  lead to look for conformal embeddings of simple affine vertex algebras in $W$-algebras. For instance, it is possible to embed conformally $V_{-\frac{n+1}{2}}(gl(n))$ in the $W$-algebra $W_{\frac{n-1}{2}}(sl(2|n),f)$ for a suitable nilpotent element in $sl(2|n)$ of even parity.
A systematic  investigation of conformal embeddings of  affine subalgebras in W-algebras has started in \cite{AKMPP1}, \cite{AKMPP2}.\par
{\bf Acknowledgments.} We would like to thank the anonymous referee for useful suggestions.
 \section{Setup and preliminary results}
\subsection{Notation}Let $\g$ be  a simple Lie algebra. Let $\h$ be a Cartan subalgebra, $\D$ the 
$(\g,\h)$-root system, $\Dp$  a set of positive roots and $\rho$ the corresponding Weyl vector. Let $(\cdot,\cdot)$ denote the  bilinear normalized invariant  form (i.e., $(\a,\a)=2$ for any long root).
If $\a\in\D$, we will denote by $X_\a$ a root vector relative to $\a$.

Assume that $\g^0$ is a reductive equal rank subalgebra of $\g$. Then $\g^0$ decomposes as 
\begin{equation}\label{decompg0}
\g^0=\g_{0}^0\oplus\g^0_1\oplus\cdots\oplus \g_t^0.
\end{equation}
 where  $\g^0_0$ is the (possibly zero) center of $\g^0$ and $\g^0_i$ are simple ideals for $i>0$. Let $\rho_0^j$ be the Weyl vector in $\g^0_j$ (w.r.t. the  set of positive roots induced by $\Dp$). 
 Assume that $(\cdot,\cdot)$ is nondegenerate when restricted to $\g^0$. Let $\p$  be the $(\cdot,\cdot)$-orthocomplement of $\g^0$ in $\g$.
 If $\mu\in \h^*$ is a $\g^0$-dominant integral weight we let $V(\mu)$ be the irreducible finite dimensional $\g^0$-module with highest weight $\mu$. Clearly we have $$V(\mu)=\bigotimes_{j=0}^tV_{\g^0_j}(\mu^{j}),$$ where $V_{\g^0_j}(\mu^{j})$ is a  irreducible $\g_j^0$-module.

 We denote by $\widehat{\g}=\C[t,t^{-1}]\otimes\g\oplus\C K\oplus\C d$ the untwisted affinization of $\g$ (see \cite[\S\ 7.2]{Kac}):  $d$ and $K$ denote, respectively, the scaling element and the canonical central element of $\widehat{\g}$.  If $x\in\g$ we set $x_{(n)}=t^n\otimes x\in\widehat{\g}$. Let $\widehat{\g^0}$ denote the subalgebra of $\ga$ generated by $\{x_{(n)}\mid x\in\g^0,\ n\in\ganz\}\cup\{d\}$. 

Let $\L_0\in(\h+\C K+\C d)^*$ be the weight such that $\L_0(K)=1$ and $\L_0(\h)=\L_0(d)=0$.  Fix $k\in\C$. We extend a  weight $\mu\in\h^*$ to $(\h+\C K+\C d)^*$ by setting $\mu(K)=\mu(d)=0$ and denote by $L_{\g} (k,\mu)$ the irreducible
highest weight $\widehat{\g}$-module with highest
weight $k\L_0+\mu$. Let $v_\mu$ be a highest weight vector in $L_{\g} (k,\mu)$. We denote by $\widetilde L_{\g^0} (k,\mu)$ the $\widehat{\g^0}$-submodule of $L_{\g} (k,\mu)$ generated by $v_\mu$ and by $L_{\g^0} (k,\mu)$ its irreducible quotient. When there is no chance of confusion we will drop $k$ from the notation denoting $L_{\g} (k,\mu)$, $\widetilde L_{\g^0} (k,\mu)$,  $L_{\g^0} (k,\mu)$ simply by $L_{\g} (\mu)$, $\widetilde L_{\g^0} (\mu)$,  $L_{\g^0} (\mu)$, respectively.

We let $V^k(\g)$,  $V_k(\g)$ denote, respectively, the universal and the simple affine vertex algebra (see \cite[\S\ 4.7 and Example 4.9b]{KacV}).
More generally, if  $\aa$ is  a reductive Lie algebra that decomposes as  $\aa=\aa_0\oplus\dots\oplus \aa_s$ with $\aa_0$ abelian and $\aa_i$ simple ideals for $i>0$ and $\mathbf{k}=(k_1,\ldots,k_s)$ is a multi-index of levels, we let
$$
V^{\mathbf{k}}(\aa)= V^{k_0}(\aa_0)\otimes\dots\otimes V^{k_s}(\aa_s),\quad
V_{\mathbf{k}}(\aa)= V_{k_0}(\aa_0)\otimes\dots\otimes V_{k_s}(\aa_s).
$$
We let $\vac$ denote the vacuum vector of both $V^{\mathbf{k}}(\aa)$ and $V_{\mathbf{k}}(\aa)$.

If $j>0$, let $\{x_i^j\},\{x_j^i\}$ be dual bases of $\aa_j$ with respect to the normalized invariant form of $\aa_j$ and $h^\vee_j$  its dual Coxeter number. For $\aa_0$, let $\{x_i^j\},\{x_j^i\}$ be dual bases with respect to any nondegenerate form and set $h_0^\vee=0$.\par
Assuming that $k_j+h_j^\vee\ne 0$ for all $j$, we consider  $V^{\mathbf{k}}(\aa)$ and all its quotients, including $V_{\mathbf{k}}(\aa)$, as  conformal vertex algebras with conformal vector $\omega_\aa$ given by the Sugawara construction:
$$\omega_\aa=\sum_{j=0}^s\frac{1}{2(k_j+h_j^\vee)}\sum_{i=1}^{\dim \aa_j} :x^j_i x_j^i :.
$$

Let  $V$ be a   vertex algebra. For  $a\in V$, we denote by $$Y(a,z)=\sum\limits_{n\in \ganz} a_{(n)}z^{-n-1}$$ the corresponding field. 
If $V$ admits a conformal vector $\omega$  then we write  the corresponding field as
$$Y(\omega,z)=\sum\limits_{n\in\ganz}\omega_nz^{-n-2}$$ (so that $\omega_{(n)}=\omega_{n-1}$). Recall that, by definition of conformal vector, $\omega_0$ acts semisimply on $V$. If $x$ is an eigenvector for $\omega_0$, then the corresponding eigenvalue $\D_x$ is called the conformal weight of $x$. 

Returning to our pair $(\g,\g^0)$, we let $\widetilde V_{k}(\g^0)$ denote the vertex subalgebra of $V_k(\g)$ generated by $x_{(-1)}\vac$, $x\in\g^0$. We choose $(\cdot,\cdot)_{|\g_0^0\times\g_0^0}$ as nondegenerate form on $\g_0^0$. Note that there is a uniquely determined multi-index $\mathbf{k}$ such that $\widetilde V_{k}(\g^0)$ is a quotient of $V^{\mathbf{k}}(\g^0)$ hence, if $k_j+h^\vee_j\ne 0$ for each $j$,  $\omega_{\g^0}$ is a conformal vector in $\widetilde V_{k}(\g^0)$. As an instance of Definition \ref{1}, we will say that $\widetilde V_{k}(\g^0)$ is conformally embedded in $V_{k}(\g)$ if 
$
\omega_\g=\omega_{\g^0}.
$
\subsection{Some general results on conformal embeddings}\label{AeP}
The basis of our investigation is the following result.
\begin{theorem}\cite[Theorem 1]{A} \label{thm-intro}
$\widetilde V_{k}(\g^0)$
is conformally embedded in
$V_k(\g)$ if and only for any $x\in\p$ we have
\bea (\omega_{\g^0})_0x_{(-1)}{\vac} = x_{(-1)} {\vac}.
\label{eqn-intro}\eea
\end{theorem}
The previous theorem has the following  useful reformulation. Remark that, in \cite{A},   it is assumed that $\g^0$ is a simple Lie algebra, but the arguments work in the reductive case as well. 

 Let $\p=\oplus_i V(\mu_i)$ be the decomposition of $\p$ as a $\g^0$-module.  Let  $( \cdot, \cdot)_0$ denote the normalized invariant
bilinear form on $\g^0$ (i.e. $( \cdot, \cdot)_0$ is the normalized invariant form when restricted to the simple ideals of $\g^0$ and, on $\g_0^0$, $(\cdot,\cdot)_0=(\cdot,\cdot)_{|\g_0^0\times \g_0^0}$).

\begin{cor}\label{numcheckk}$ \widetilde{V}_{k}(\g^0)$ is
conformally embedded in $V_k({\g})$ if and only if
\bea\label{numcheck} && \sum_{j=0}^t\frac{ ( \mu_{i}^j , \mu _{i}^j + 2 \rho_0^j ) _0}{ 2 (k_j + h_j
^{\vee} )} = 1  \eea
for any $i$.
\end{cor}

We now discuss the special case when $\g^0$ is semisimple and it is the fixed
point subalgebra of an automorphism $\sigma$ of $\g$ of order $m$. Let $\xi$ be a
 primitive $m$-th root of unity. Since $\g^0$ is semisimple, the eigenspace associated to the
eigenvalue $\xi ^i$ (for $i=1, \ldots ,m-1$) is an irreducible $\g^0$-module $V(\mu_i)$ and $\p=\oplus_{i=1}^{m-1} V(\mu_i)$.  Set also $\mu_0=0$.

The automorphism $\sigma$ can be extended
to a finite order automorphism of the simple vertex  algebra
$V_{k} (\g)$ which admits the following decomposition
$$V_{k} (\g) = V_{k} (\g)^0 \oplus V_{k} (\g)^1 \oplus \cdots \oplus V_{k} (\g)^{m-1},$$
where
$$V_{k} (\g)  ^{i} = \{ v \in V_{k} (\g) \ \vert \ \sigma (v) = \xi ^i v \}. $$
Clearly $V_{k} (\g) ^{i}$ is a $\widehat{\g^0}$--module. 

%

\begin{theorem}\cite[Theorem 3]{A} \label{general}
Assume that
\bea
&&
V(\mu_i) \otimes V(\mu_j) = V(\mu_l) \oplus  \bigoplus_{r=1} ^{m_{i,j}} V (\nu_{r,i,j}) \label{uvj-1} ;\\
&&V_k(\g)^{l} \ \mbox{does not contain} \
\widehat{\g^0}-\mbox{primitive vectors of weight} \
\nu_{r,i,j}, \label{uvj-2} \\
&& ( \mbox{where} \ l = i + j \, \mod m \ ), \nonumber
 \eea
for all $i,j \in \{1, \dots, m-1\}$ and $r=1, \dots, m_{i,j}.$

Then
$$V_k(\g) =L_{\g^0}(0) \oplus L_{\g^0} ( \mu_1) \oplus \cdots \oplus L_{\g^0} (\mu_{m-1}). $$
\end{theorem}

The next result is a version of Theorem \ref{general} suitable for the case when $\g^0$ is  a reductive equal rank non-semisimple subalgebra of $\g$. We also assume here that the embedding $\g^0\subset \g$ is maximal.  In this case  $\dim\g^0_0=1$, and $\g^0$ is the fixed point set of an automorphism of $\g$ of order $2$.
We can also choose $\Dp$ so that the set of simple roots of $\g^0$ is obtained from the set of simple roots of $\g$ by dropping one simple root $\a_p$ having coefficient $1$ in the simple roots expansion of the highest  root  $\theta$ of $\g$. 

Moreover, as a $\g^0$-module, $\p=V(\mu_1)\oplus V(\mu_2)$ with $\mu_1=\theta$ and $\mu_2=-\a_p$. In particular, $\mu_{1}^0=\theta_{|\g^0_0}$, and $\mu_{2}^0=-(\alpha_p)_{|\g^0_0}$. Since $(\cdot,\cdot)$ is nondegenerate on $\p$ and it is also $\g^0$-invariant, we have that $V(-\a_p)=V(\theta)^*$. In particular the trivial representation occurs in $V(\theta)\otimes V(-\a_p)$.
Thus we can write
\begin{equation}
V(\theta) \otimes V(-\a_p) = \C \oplus  \bigoplus_{r=1} ^{s} V (\nu_{r}) \label{uvj-1hs}.
\end{equation}
with $\nu_i\ne 0$.
Let $\varpi$ be the
element of $\h$ such that $\a_i(\varpi)=\delta_{ip}$. Then $\g^0_0=\C\varpi$.  Let $\zeta\in(\g^0_0)^*$ be defined by setting $\zeta(\varpi)=1$, so that 
\begin{equation}\label{mu12}
\mu_{1}^0=\zeta,\quad\mu_{2}^0=-\zeta.
\end{equation}
If $q\in\ganz$, let $V_k(\g)^{(q)}$ be the eigenspace for the action of $\varpi_{(0)}$ on $V_k(\g)$ relative to the eigenvalue $q$. 

If $A$, $B$ are subspaces of a vertex algebra $V$, we set   
\begin{equation}\label{fp}
A\cdot B = span \{a_{(n)}b\mid a\in A, b\in B, n\in\ganz\}.
\end{equation}

\begin{theorem}\label{generalhs}
Assume $k\ne 0$ and that
\bea
&&
V(\theta) \otimes V(-\a_p) = \C \oplus  \bigoplus_{r=1} ^{s} V (\nu_{r}) \label{uvj-hs}; \\
&&V_k(\g)^{(0)} \ \mbox{does not contain} \
\widehat{\g^0}-\mbox{primitive vectors of weight} \
\nu_{r}, \label{uvj-2hs} \\
&& ( \mbox{where} \ r= 1,\ldots,s ), \nonumber
 \eea

Then $\widetilde V_k(\g^0)\cong V_\mathbf{k}(\g^0)$ and as a $\widehat{\g^0}$-module,
\begin{equation}\label{eigenweight}
V_k(\g)^{(q)}=\begin{cases}L_{\g^0}(q\theta)&\text{if $q\ge 0$,}\\
L_{\g^0}(q\a_p)&\text{if $q\le 0$.}
\end{cases}
\end{equation}
In particular, 
\begin{equation}\label{efd}
V_k(\g) =L_{\g^0}(0)\oplus(\sum_{q>0} L_{\g^0}(q\theta) )\oplus(\sum_{q>0} L_{\g^0} ( -q\a_p).
\end{equation}
\end{theorem}
\begin{proof}
Let $A^+$, $A^-$  be the $\widehat{\g^0}$-submodules of $V_k(\g)$ generated by $V(\theta)_{(-1)}\vac$, $V(-\a_p)_{(-1)}\vac$ respectively. 

Then a fusion rules argument shows that  a primitive  vector in $A^+\cdot A^-$ must have weight $0$ or $\nu_r$ for some $r$.  By our hypothesis it has weight $0$ so, since the embedding is conformal, it has conformal weight $0$. Since the only vector of conformal weight $0$ in $V_k(\g)$ is $\vac$, we see that 
\begin{equation}\label{fusiontrivial}
A^+\cdot A^-\subset\widetilde V_k(\g^0).
\end{equation}

It is clear that $V_k(\g)^{(0)}$ is contained in the sum of all products 
of type $A_1\cdot(A_2\cdot(\cdots\cdot A_r))\cdots)$ with $A_i\in\{A^+,A^-,\widetilde V_k(\g^0)\}$ such that
$$
\sharp\{i\mid A_i=A^+\}=\sharp\{i\mid A_i=A^-\}.
$$
 By the associativity of the $\cdot$ product  \eqref{fp} (see \cite[Remark 7.6]{BK2}) we see that \eqref{fusiontrivial} implies that $A_1\cdot(A_2\cdot(\cdots\cdot A_r))\cdots)\subset\widetilde V_k(\g^0)$, so $V_k(\g)^{(0)}=\widetilde V_k(\g^0)$. 
It follows that $\widetilde V_k(\g^0)$  is simple, hence isomorphic to $V_\mathbf{k}(\g^0)$,  and $V_k(\g)^{(q)}$ is a simple $V_\mathbf{k}(\g^0)$-module for all $q$.

It remains to prove relation \eqref{eigenweight}. To do this, we check  that $({X_\theta}_{(-1)})^q\vac$ and $({X_{-\a_p}}_{(-1)})^q\vac$ are nonzero singular vectors in $V_k(\g)$ for all $q\in\nat$. 
We first verify that they are singular for $\widehat{\g^0}$ and then that they are nonzero.
It is easy to see that
\begin{equation}\label{123}
\varpi_{(i)}({X_\theta}_{(-1)})^q\vac=0,\quad\varpi_{(i)}({X_{-\a_p}}_{(-1)})^q\vac=0
\end{equation}
for all $i>0$. E.g., using relation 
\begin{align*}
\varpi_{(i)}({X_\theta}_{(-1)})^q\vac&={X_\theta}_{(i-1)}({X_\theta}_{(-1)})^{q-1}\vac+{X_\theta}_{(-1)}\varpi_{(i)}({X_\theta}_{(-1)})^{q-1}\vac\\
&={X_\theta}_{(-1)}\varpi_{(i)}({X_\theta}_{(-1)})^{q-1}\vac,
\end{align*}
an obvious induction gives the leftmost formula in \eqref{123}.\par
It is also clear that, if $\a_i\ne\a_p$, then
$$
{X_{\a_i}}_{(0)}({X_\theta}_{(-1)})^q\vac={X_{\a_i}}_{(0)}({X_{-\a_p}}_{(-1)})^q\vac=0.
$$
If $\be$ is the highest root of a simple ideal of $\g^0$ and $\theta-\be$ is not a root then ${X_{-\be}}_{(1)}({X_\theta}_{(-1)})^q\vac=0$, while, if $\gamma=\theta-\beta\in\Dp$ then ${X_{-\be}}_{(1)}({X_\theta}_{(-1)})^q\vac=({X_\gamma}_{(0)})({X_\theta}_{(-1)})^{q-1}\vac=0$. 

Recall that $\a_p$ has coefficient $1$ in the expansion of $\theta$ in terms of simple roots. Hence $\beta+2\a_p\notin \D$ and the same argument as above shows that
$
{X_{-\be}}_{(1)}({X_{-\a_p}}_{(-1)})^q\vac=0$.\par  We now prove   by induction on $q$ that  $({X_\theta}_{(-1)})^q\vac$ and $({X_{-\a_p}}_{(-1)})^q\vac$ are nonzero, the base $q=0$ being obvious. \par
Assume by induction that $({X_\theta}_{(-1)})^{q-1}\vac\ne 0$. Then, since
\begin{align*}
{X_{-\theta}}_{(1)}({X_\theta}_{(-1)})^q\vac&=-{h_\theta}_{(0)}({X_\theta}_{(-1)})^{q-1}\vac+k({X_\theta}_{(-1)})^{q-1}\vac\\&=(-2(q-1)+k)({X_\theta}_{(-1)})^{q-1}\vac,
\end{align*}
we have that $({X_\theta}_{(-1)})^q\vac$ can be $0$ only if $k=2(q-1)$. An embedding in an integrable module can be conformal only when $k=1$
(see e.g. \cite{AGO}), thus $({X_\theta}_{(-1)})^q\vac\ne 0$. Computing ${x_{\a_p}}_{(1)}({X_{-\a_p}}_{(-1)})^q\vac$, we see likewise that $({X_{-\a_p}}_{(-1)})^q\vac\ne 0$.
\end{proof}
 \begin{rem}\label{conditionsfinite}
 Condition \eqref{uvj-2} holds whenever 
\bea\label{finiote}
 \sum_{u=0}^t\frac{ ( \nu_{r,i,j}^u , \nu _{r,i,j}^u + 2 \rho_{0}^u ) _0}{ 2 (k_u+ h^\vee _{u})}\not\in \ganz_+,
\eea
for any $i,j$, while  condition \eqref{uvj-2hs} holds whenever 
\bea\label{finioths}
 \sum_{u=0}^t\frac{ ( \nu_{r}^u , \nu _{r}^u + 2 \rho_{0}^u ) _0}{ 2 (k_u+ h^\vee _{u})}\not\in \ganz_+
\eea
for any $r$.
 \end{rem}
 \subsection{Dynkin indices and combinatorial formulas}
We now review some results by Panyushev \cite{PLMS}, which will be used in the proof of Theorem \ref{main}.
Recall that  $\g$ is simple and $\g^0$ is a reductive equal rank subalgebra of $\g$ with decomposition as in \eqref{decompg0}.
 We denote by $\kappa(\cdot,\cdot)$ the Killing form of $\g$. Recall that  $\kappa(\cdot,\cdot)=2h^\vee\,(\cdot,\cdot)$, hence, denoting with same symbol the bilinear forms induced on $\h^*$, we have 
 \begin{equation}\label{cx}
 \kappa(\l,\mu)=\frac{1}{2h^\vee}\,(\l,\mu)\ \ \forall\,\l,\mu\in\h^*.
 \end{equation}
   Set $d_0=1$ and for $j>0$
\begin{equation}\label{d}
d_j={\frac{2}{(\theta_j,\theta_j)}},
\end{equation} 
where $\theta_j$ is the highest root   of $\g^0_j$.

 If $V$ is  a finite-dimensional $\g$-module and $x\in\g$,  let $\pi(x)\in End(V)$ denote the action of $x$ on $V$. The trace form $(x,y)_V:=tr(\pi(x)\pi(y))$ is an invariant form on $\g$, hence there is $d_V\in\C$ such that $d_V(\cdot,\cdot)_V=\kappa(\cdot,\cdot)$. The number $d_V$ is called the Dynkin index of $V$ and denoted by $ind_\g(V)$.
 The Dynkin index is clearly additive with respect to direct sums and, if $V_\l$ is the irreducible finite-dimensional $\g$-module with highest weight $\l$, then
  \begin{equation}\label{ind}
 ind_\g(V_\l)=\frac{\dim V_\l}{\dim \g}\,\frac{(\l,\l+2\rho)}{(\theta,\theta+2\rho)}.
 \end{equation}
We write \eqref{ind} in terms of the normalized invariant form, but the formula is clearly independent from the choice of the form. Let $C_\g$ denote the Casimir element of $\kappa$ and  $C_{\g^0_j}$ the Casimir element of $\kappa_{|\g^0_j\times \g^0_j}$.

  \begin{prop}\cite[Proposition 2.2]{PLMS}�\label{p}\
\begin{enumerate} 
\item The eigenvalue of $C_{\g^0_j}$ on $\g^0_j$ is $\frac{h_j^\vee}{d_jh^\vee}$. (Recall that $h_0^\vee=0$.)
\item $ind_{\g^0_j}(\p)=\frac{d_jh^\vee}{h_j^\vee}-1,\,\,j>0$.
\end{enumerate}
\end{prop}
 \begin{prop}\cite[Corollary 2.7]{PLMS}\label{eigen}� Assume that $\g^0$ is semisimple and that it is the fixed-point  set of an automorphism of $\g$ of prime order $m$. Then $C_{\g^0}$ acts scalarly on $\p$
 with eigenvalue $1/m$.
 Also, if $m=2$, the above statement holds with $\g^0$ reductive.\end{prop}

 \section{A criterion for conformality}\label{tre}
 Let $k\ne -h^\vee$. For a simple or abelian Lie algebra $\g$ set 
\begin{equation}\label{cc}
c_\g(k)=\frac{k\dim\g}{k+h^\vee}.
\end{equation}
If $\aa$ is a reductive Lie algebra, which decomposes as  $\aa=\aa_0\oplus\dots\oplus \aa_s$ then we set, for a multindex $\mathbf{k}=(k_0,\ldots,k_s)$,
\begin{equation}\label{ccc}
c_\aa(\mathbf{k})=\sum\limits_{j=0}^s c_{\aa_j}(k_j).
\end{equation}
\begin{theorem}\label{main}  Let $\g$ be a simple Lie algebra and $\g^0$ a maximal equal rank reductive  subalgebra. Then $\widetilde V_{k}(\g^0)$ is a conformal subalgebra of 
$V_{k}(\g)$ if and only if 
\begin{equation}\label{ecc}c_{\g}(k)= c_{\g^0}(\mathbf{k}). \end{equation}
\end{theorem}
\begin{proof} The statement is trivially verified when $k=0$, so we can assume $k\ne 0$. Recall that $c_\g(k)$ is the central charge of the conformal vector of  $V_{k}(\g)$, and that $\sum\limits_{j=0}^t c_{\g^0_j}(k_j)$ is the central charge  of the conformal vector of
$\widetilde V_{k}(\g^0)$. Hence, if  $\widetilde V_{k}(\g^0)$ is a conformal subalgebra of 
$V_{k}(\g)$, equality \eqref{ecc} holds.\par
To prove the converse, by the classification of finite order automorphisms of simple Lie algebras \cite{Kac} and Borel--de Siebenthal Theorem \cite{bds},  $\g^0$ is the fixed point subalgebra of an automorphism of $\g$ of finite order $m$, which is a prime number by maximality (indeed $m=2,3,5$). Recall  that if $\g^0_0\ne \{0\}$, then $\dim\g^0_0=1$ and $m=2$.

Recall that, as a $\g^0$-module, $\p=\oplus_{i=1}^s V(\mu_i)$ with $\mu_i=\sum_{j=0}^t\mu_{i}^j$. To prove our claim, we verify relation \eqref{numcheck}.  We have to estimate $( \mu_{i}^j , \mu_{i}^j + 2 \rho_0^j )_0$. 

 We start discussing the cases when $m$ is either $2$ or $3$.  By the additivity of the index, combining \eqref{ind} with part (2) of Proposition \ref{p} (and using the conversion formulas \eqref{cx}, \eqref{d}) we get, if $j>0$,
$$\frac{d_j h^\vee}{h_j^\vee}-1=\sum_{i=1}^{s}\frac{\dim V(\mu_i)}{\dim\g^0_j}\,\,\frac{( \mu_{i}^j , \mu_{i}^j + 2 \rho_0^j )_0}{(\theta_j,\theta_j+2\rho_0^j)_0}.$$

 If $m=2,3$, then $s=1$ or $s=2$ and $V(\mu_2)=V(\mu_1)^*$.   It follows that, in both cases, $\dim V(\mu_i)$ as well as $( \mu_{i}^j , \mu_{i}^j + 2 \rho_0^j )_0$ are independent of $i$, so $\dim V(\mu_i)=\frac{\dim\p}{s}$ and
$$\frac{d_j h^\vee}{h_j^\vee}-1=\frac{\dim \p}{\dim\g^0_j}\,\,\frac{( \mu_{i}^j , \mu _{i}^j + 2 \rho_0^j )_0}{(\theta_j,\theta_j+2\rho_0^j)_0}.$$
Since $(\theta_j,\theta_j+2\rho_0^j)_0=2h^\vee_j$, we have

\begin{equation}\label{2}
( \mu_{i}^j , \mu_{i}^j + 2 \rho_0^j )_0=\frac{2\dim\g^0_j}{\dim\p}(d_jh^\vee-h_j^\vee).
\end{equation}

Let us discuss the case $j=0$. Assume that $\g^0_0\ne\{0\}$. We need to compute $(\mu_{i}^0,\mu_{i}^0)$ i.e. $(\zeta,\zeta)$ (cf. \eqref{mu12}).
Since the roots of $\p$ are precisely the roots $\a$ such that $|\a(\varpi)|=1$, we have $\kappa(\varpi,\varpi)=\dim \p$, hence $h_\zeta=\frac{1}{\dim\p}\varpi$ is the unique element of $\h$ such that $\kappa(h_\zeta,h)=\zeta(h)$. Hence $\kappa(\zeta,\zeta)=\kappa(h_\zeta,h_\zeta)=\frac{1}{\dim\p}$. It follows form the relation $\kappa(\cdot,\cdot)=\frac{1}{2h^\vee}(\cdot,\cdot)$ that $(\zeta,\zeta)=\frac{2h^\vee}{\dim\p}$, so \eqref{2} holds also in this case.
\par
Now we proceed to evaluate \eqref{cc}.
Setting $g=\dim\g,\,g_i=\dim\g^0_i,\,i=0,\ldots,t$, we may write  
\eqref{cc} as
\begin{equation}\label{g}
\frac{gk}{k+h^\vee}-\sum_{j=0}^t\frac{d_jg_jk}{d_jk+h_j^\vee}=0.
\end{equation}
Multiplying \eqref{g} by $\frac{k+h^\vee}{k}$ we find that
\begin{align*}
0&=g-\sum_{j=0}^t\frac{d_jg_j(k+h^\vee)}{d_jk+h_j^\vee}=g-\sum_{j=0}^t g_j +\sum_{j=0}^t g_j-\sum_{j=0}^t\frac{d_jg_j(k+h^\vee)}{d_jk+h_j^\vee}\\
&=\dim\p +\sum_{j=0}^t\left(g_j-\frac{d_jg_j(k+h^\vee)}{d_jk+h_j^\vee}\right)=\dim\p-\sum_{j=0}^t\frac{ g_j(d_jh^\vee-h^\vee_j)}{  d_jk+ h _{j}
^{\vee} }.
\end{align*}
It follows that
\begin{equation}
 1=\sum_{j=0}^t\frac{ 2g_j(d_jh^\vee-h^\vee_j)}{2\dim\p  (d_jk+ h _{j}
^{\vee} )}
\end{equation}
and, using \eqref{2}, we find 
\begin{equation}\label{4}
 \sum_{j=0}^t\frac{ ( \mu_{i}^j , \mu_{i}^j + 2 \rho_0^j ) _0}{ 2 (d_jk+ h _{j}
^{\vee} )} = 1.
\end{equation}

In the remaining $m=5$ case we have that $\g$ is of type $E_8$ and $\g^0$ is of type $A_4\times A_4$. In this case the equality of central charges reads
$$
\frac{248}{k+30}=\frac{48}{k+5}.
$$
so that $k=1$.
Formula \eqref{numcheck} reduces to
$$
\frac{1}{12}\sum_{j=1}^2(\mu_{i}^j,\mu_{i}^j+2\rho_0^j)_0=1.
$$
This is readily checked since, by Proposition \ref{eigen}, we have
$
\sum_{j=1}^2\kappa(\mu_{i}^j,\mu_{i}^j+2\rho_0^j)=\frac{1}{m}
$
hence
$
\sum_{j=1}^2(\mu_{i}^j,\mu_{i}^j+2\rho_0^j)_0=\frac{2h^\vee}{m}=12
$.
\end{proof}

We now discuss the extension of Theorem \ref{main} to the embeddings of an equal rank subalgebra $\g^0$ in $\g$ which is not maximal. 
We start with the following
\begin{lemma} Let $\g^0$ be a maximal conformal subalgebra of $\g$ (i.e., maximal among conformal subalgebras of $\g$). Then $\g^0$ is a maximal reductive subalgebra of $\g$.
\end{lemma}
\begin{proof}
Suppose by contradiction  that  there exists a reductive subalgebra $\k$ of $\g$ with $\g^0\subsetneq \k\subsetneq \g$ . Then, since the form $(\cdot,\cdot)$ is nondegenerate when restricted to $\k$, the orthocomplement $\p$  of $\g^0$ in $\g$ can be written as $\p=\p\cap\k\oplus V$ with $V$ the orthocomplement of $\k$ in $\g$. If the embedding of $\g^0$ in $\g$ is conformal, then, by Theorem \ref{thm-intro}, $(\omega_{\g^0})_0x_{(-1)}\vac=x_{(-1)}\vac$ for all $x\in \p\cap\k$. It follows that the embedding of $\g^0$ in $\k$ is conformal, hence, clearly, also the embedding of $\k$ in $\g$ is conformal.
\end{proof}
 This observation leads to the following criterion.
\begin{prop}\label{critnonmax}
 Let $\g^0=\k_1\subset \k_2\subset\dots \subset\k_t=\g$ be a sequence of equal rank subalgebras with $\k_i$ maximal in $\k_{i+1}$. Let $\mathbf{k}_i$ be the multi-index such that the vertex subalgebra spanned in $V_k(\g)$ by $\{x_{(-1)}\vac\mid x\in\k_i\}$ is a quotient of $V^{\mathbf{k}_i}(\k_i)$. 
 Then $\g^0\subset \g$ is a conformal embedding if and only if \begin{equation}\label{non maximal}
 c_{\k_i}(\mathbf k_i)=c_{\k_{i-1}}(\mathbf k_{i-1}),\ i=1,\ldots,t.
 \end{equation}
\end{prop}
\begin{proof}
If $\g^0\subset \g$ is a conformal embedding, then, by the above lemma, $\k_{i-1}\subset\k_i$ is a conformal embedding. Since the two subalgebras contain a Cartan subalgebra of $\g$, we have that there are ideals $\ov\k_i,\k'_{i}$ of $\k_i$ with $\ov\k_i$ simple or abelian, and an ideal $\ov \k_{i-1}$ of $\k_{i-1}$ such that $\k_i=\ov\k_{i}\oplus \k'_{i}$, $\k_{i-1}=\ov\k_{i-1}\oplus \k'_{i}$ and 
$
\ov \k_{i-1}\subset \ov\k_i$ maximal embedding. Applying Theorem \ref{main} to the embedding $
\ov \k_{i-1}\subset \ov\k_i$, we obtain $c_{\k_i}(\mathbf{k}_i)=c_{\k_{i-1}}(\mathbf{k}_{i-1})$. 

If condition \eqref{non maximal} holds, then, by Theorem \ref{main}, the embedding $
\ov \k_{i-1}\subset \ov\k_i$ is conformal, hence the embedding $\k_{i-1}\subset\k_i$ is conformal for $i=1,\ldots,t$, so the embedding $\g^0\subset \g$ is conformal.
\end{proof}

\begin{defi}
We say that a  level $k$ of $V_{k}(\g)$  for which equality 	\eqref{ecc} holds is a  {\it conformal level} of the pair $(\g,\g^0)$.
\end{defi}
In the following tables we list the conformal levels for all reductive maximal equal rank subalgebras of simple Lie algebras, thus classifying all possible maximal conformal equal rank embeddings. The cases where $\g^0$ is not semisimple are denoted by $X\times Z$ where $X$ is the type of $[\g^0,\g^0]$ and  $Z$ denotes the one-dimensional center of $\g^0$.

\begin{center}

\centerline{Type $A_n,\,n\geq 1$}\vskip5pt
\begin{tabular}{c|c}
$\g^0$& conformal level\\
\noalign{\smallskip}\hline\noalign{\smallskip}
$A_h\times A_{n-h-1}\times Z$, $h\ge 1$, $n-h\ge 2$&$1; -1; -\frac{n+1}{2} \text{ (if $h\ne (n-1)/2)$}$\\
\noalign{\smallskip}\hline\noalign{\smallskip}
$A_{n-1} \times Z$, $n\ge1$&$1; -\frac{n+1}{2} \text{ (if $n>1$)}$\\
\end{tabular}
\end{center}

\begin{center}
\centerline{Type $D_n$, $n\ge 4$}\vskip5pt

\begin{tabular}{c|c}
$\g^0$&conformal levels\\
\noalign{\smallskip}\hline\noalign{\smallskip}
$D_h\times D_{n-h}$, $h\ge 2$, $n-h\ge 2$&$1; 2-n \text{ (if $h\ne n/2$)}$\\
\noalign{\smallskip}\hline\noalign{\smallskip}
$D_{n-1}\times Z$&$1; 2-n$\\
\noalign{\smallskip}\hline\noalign{\smallskip}
$A_{n-1}\times Z$&$1; -2$\\
\end{tabular}
\end{center}
\vskip10pt

\begin{center}

\centerline{Type $C_n$, $n\ge 2$}\vskip5pt
\begin{tabular}{c|c}
$\g^0$&conformal levels\\
\hline\noalign{\smallskip}
$C_h\times C_{n-h}$, $h\ge 1$, $n-h\ge 1$&$-\frac{1}{2}; -1-\frac{n}{2} \text{ (if $h\ne n/2$)}$\\
\noalign{\smallskip}\hline\noalign{\smallskip}
$A_{n-1} \times Z$&$1; -\frac{1}{2}$
\end{tabular}
\end{center}

\begin{center}
\centerline{Type $B_n$, $n\ge 3$}\vskip5pt
\begin{tabular}{c|c}
$\g^0$&conformal levels\\
\noalign{\smallskip}\hline\noalign{\smallskip}
$D_h\times B_{n-h}$, $h\ge 1$, $n-h\ge 1$&$1; \frac{3}{2}-n$ \\
\noalign{\smallskip}\hline\noalign{\smallskip}
$D_n$ &$\frac{3}{2}-n$ \\\noalign{\smallskip}\hline\noalign{\smallskip}
$B_{n-1}\times Z$, $n\ge 4$&$1; \frac{3}{2}-n$
\end{tabular}
\end{center}

\begin{rem}\label{41} Assume $\g^0$ semisimple and let $\g^0=\g_1^0\oplus\ldots\oplus\g^0_t$ be its decomposition into simple ideals. 
Recall that a set of simple roots for $\g^0$ can be obtained from $\Pi\cup\{-\theta\}$ by removing
a simple root of $\Pi$. Call this root $\a_p$, and set 
$$\xi=\begin{cases}-\frac{1}{2}&\quad\text{if $\a_p$ is short,}\\ 1 &\quad\text{otherwise.}\end{cases}$$
Then the previous tables show that equality of central charges occurs at the following levels:
\begin{itemize}
\item $k=-\frac{h^\vee}{2}+1$ if $t=1$;
\item $k\in\{\xi,-\frac{h^\vee}{2}+\xi\}$ if $t>1$ and there are at least two non-isomorphic simple ideals or all simple ideals are isomorphic but two of them have different Dynkin index;
\item $k=\xi$ otherwise, i.e., all simple ideals are isomorphic and  have the same Dynkin index.
\end{itemize}

\end{rem}
\begin{center}
\centerline{Type $E_6$}\vskip5pt
\begin{tabular}{c|c}
$\g^0$&conformal levels\\
\hline
$A_{1}\times A_5$&$1; -3$\\
\hline
$A_{2}\times A_2 \times A_2$ &$1$\\
\hline
$D_{5}\times Z$&$1; -3$
\end{tabular}
\end{center}

\begin{center}
\centerline{Type $E_7$}\vskip5pt
\begin{tabular}{c|c}
$\g^0$&conformal levels\\
\hline
$A_{1}\times D_6$&$1; -4$\\
\hline
$A_{2}\times A_5$&$1; -4$\\
\hline
$A_7$&$1$\\
\hline
$E_6\times Z$&$1; -4$
\end{tabular}
\end{center}
\vskip10pt

\begin{center}
\centerline{Type $E_8$}\vskip5pt
\begin{tabular}{c|c}
$\g^0$&conformal levels\\
\hline
$A_{1}\times E_7$&$1; -6$\\
\hline
$A_{2}\times E_6$&$1; -6$\\\hline
$A_{4}\times A_4$&$1$\\\hline
$D_8$&$1$\\
\hline
$A_8$&$1$\\
\end{tabular}
\end{center}
\par
\begin{center}
\centerline{Type $F_4$}\vskip5pt
\begin{tabular}{c|c}
$\g^0$&conformal levels\\
\hline
$A_{1}\times C_3$&$1; -\frac{5}{2}$ \\
\hline
$A_{2}\times A_2$&$1; -\frac{5}{2}$ \\\hline
$B_4$&$-\frac{5}{2}$\\
\end{tabular}
\end{center}
\begin{center}
\centerline{Type $G_2$}\vskip5pt
\begin{tabular}{c|c}
$\g^0$&conformal levels\\
\hline
$A_{1}\times A_1$&$1; -\frac{5}{3}$\\
\hline
$A_{2}$&$-\frac{5}{3}$
\end{tabular}
\end{center}

\begin{rem}
Recall the Deligne exceptional series \cite{D}, \cite{Kawa}
$$A_1\subset A_2\subset  G_2 \subset D_4\subset  F_4 \subset E_6 \subset E_7 \subset E_8.$$
Set $\aa=Span_\C\{X_\beta\mid (\beta,\theta)=0\} + h_\theta^\perp$. We remark that if $\g\not\cong A_1, D_4$ is in the Deligne  series, then $\aa$ is simple or abelian. Moreover, there is always a maximal equal rank conformal  subalgebra of the form $sl(2)\times \aa$ where $sl(2)=\C X_{\theta}\oplus\C h_{\theta} \oplus \C X_{-\theta}$. For this maximal subalgebra,  the conformal levels are $1$ and $-h^\vee/6-1$.
This numerological coincidence can be explained as follows. 
Let $W_{min}(k)$ be the minimal simple W-algebra of level $k$ for $\g$.
By \cite[Theorem 7.2]{AM}, $W_{min}(-h^\vee/6-1)$ is 1-dimensional. On the other hand,  by \cite[Theorem 5.1(d)]{KW}, $W_{min}(k)$    has a subalgebra of currents $\widehat\aa$
of level $k+(h^\vee-h^\vee_\aa)/2$, where $h^\vee$ is Coxeter number of $\g$ and $h^\vee_\aa$ that of $\aa$. Hence this level should be $0$ when $k=-h^\vee/6-1$, therefore
\begin{equation}\label{12345}-\frac{h^\vee}{6}-1=-\frac{h^\vee-h^\vee_\aa}{2}.\end{equation}
Now recall that, for $\g$ in the Deligne series, we have $\dim\g=\frac{2(h^\vee+1)(5h^\vee-6)}{h^\vee+6}$ \cite{D}. Since the number of roots not orthogonal to $\theta$ is $4h^\vee-6$ \cite{S}, one easily verifies that the degree 3 equation $c_k(\g)=c_k(\aa)+c_k(sl(2))$ in the variable $k$ has exactly $1$ and $-h^\vee/6-1$ as nonzero roots.
\end{rem}

\section{Finite decomposition for maximal equal rank semisimple embeddings}\label{FD}
In this section we determine precisely, for all pairs $(\g,\g^0)$ with $\g^0$ semisimple,   the  conformal levels  such that  the decomposition is finite. The main result is the following theorem, which will be proved along the section.  
\begin{theorem}\label{mainss} If  $\g^0$ is semisimple, the conformal levels different from $1$  for which  $V_k(\g)$ decomposes  as a finite sum of $\widehat{\g^0}$-irreducibles
are the following:
$$
\begin{tabular}{ c | l | l  }
conformal level& $\g$ & $\g^0$ \\ \hline
$-1/2$&$C_n$&$C_h\times C_{n-h}$\\
  $-n+3/2$ & $B_n$ & $D_n$ \\
  $-5/3$ & $G_2$ & $A_2$ \\
  $-5/2$ & $F_4$ & $B_4$ \\
\end{tabular}
$$
In type $C_n$,  $h$ ranges from $1$ to $n-1$;  by  $C_1$ we mean $A_1$. In type $B_3$ we set $D_3=A_3$.
\end{theorem}
\vskip10pt

We start the proof of Theorem \ref{mainss} with a simple computation.

 \begin{lemma}\label{actionxeta}Let $\eta$ be a highest weight of $\p$. Set $v_m=(X_\eta)^m_{(-1)}\vac$. Then
$$
(X_{-\eta})_{(1)}v_{m+1}=((m+1)k-\Vert\eta\Vert^2\binom{m+1}{2})v_{m}.
$$
\end{lemma}
\begin{proof}
We prove the formula by induction on $m$. Assume that $(X_\eta,X_{-\eta})=1$. If $m=0$,
$$
(X_{-\eta})_{(1)}(X_\eta)_{(-1)}\vac=-(h_\eta)_{(0)}\vac+k\vac=kv_0.
$$

If $m>0$,
$$
(X_{-\eta})_{(1)}v_{m+1}=-(h_\eta)_{(0)}v_{m}+kv_{m}+(X_\eta)_{(-1)}(X_{-\eta})_{(1)}v_{m}.
$$
hence
\begin{align*}
(X_{-\eta})_{(1)}v_{m+1}&=-m\Vert\eta\Vert^2v_{m}+kv_{m}+(mk-\Vert\eta\Vert^2\binom{m}{2})v_{m}\\
&=((m+1)k-\Vert\eta\Vert^2\binom{m+1}{2})v_{m}.
\end{align*}
\end{proof}

 Let $\g^0$ be semisimple. Let $\a_p$ be as in Remark \ref{41}.

\begin{prop}\label{43}
If $\a_p$ is long then the decomposition is finite if and only if $k=1$.
\end{prop}
\begin{proof}
By \cite{Kac}, the decomposition is finite if $k=1$ and $\a_p$ is long. 
Assume now that the decomposition is finite. Since $\a_p$ is a weight of $\p$, there is $\mu_i$ in the decomposition $\p=\oplus_jV(\mu_j)$ such that 
$\a_p$ occurs as a weight of $V(\mu_i)$. Set $\eta=\mu_i$. Since $\a_p$ is a long root, the same holds for $\eta$. Let $\theta_j$ be the highest root of  $\g^0_j$. Since $\eta$ is the highest weight of $\p$, $\eta+\theta_j\not\in\D$. Since $\eta$ is long $|(\theta_j,\eta^\vee)|\le 1$. This implies that $-\theta_j+2\eta\not\in \D$. Since $-\theta_j-\eta$ is not a root, the $\eta$-root string through  $-\theta_j$ is $-\theta_j,\ldots, -\theta_j+q\eta$. Since $-q=-(\theta_j,\eta^\vee)$, we see that $q\le 1$. 

We now prove by induction on $m$, that the vectors $v_m$ of Lemma \ref{actionxeta} are $\widehat{\g^0}$-singular. If $m=0$ this is obvious.
Let $\D_0$ be the $(\g^0,\h)$-root system and let $\Dp_0=\Dp\cap\D_0$ be a set of positive roots.
It is clear that if $\a\in \Dp_0$ then $\eta+\a$ is not a root (recall that $\eta$ is a highest weight in $\p$), so $[(X_\a)_{(0)},(X_\eta)_{(-1)}]=0$. This implies that, if $m>0$,   $(X_\a)_{(0)}v_m=0$, so we need only to check that $(X_{-\theta_j})_{(1)}v_m=0$ for all $j$.
Clearly
\begin{align*}(X_{-\theta_j})_{(1)}v_m&=[X_{-\theta_j},X_\eta]_{(0)}v_{m-1}+(X_\eta)_{(-1)}(X_{-\theta_j})_{(1)}v_{m-1}\\&=[X_{-\theta_j},X_\eta]_{(0)}v_{m-1}.
\end{align*}
The last equality follows from the induction hypothesis.
Since $-\theta_j+2\eta$ is not a root we see that
$[[X_{-\theta_j},X_\eta],X_\eta]=0$. This implies that 
$$
[X_{-\theta_j},X_\eta]_{(0)}v_{m-1}=(X_\eta)^{m-1}_{(1)}[X_{-\theta_j},X_\eta]_{(0)}\vac=0
$$
as desired. Since the decomposition is finite, the vectors $v_m$ must span a finite dimensional space.  Since they have different weights, they are independent if nonzero, so they are almost all zero. Let  $M$ be such that $v_M\ne0$ and $v_m=0$ for $m>M$. By Lemma \ref{actionxeta}, since $\Vert\eta\Vert^2=2$,
$$
0=(X_{-\eta})_{(1)}v_{M+1}=(M+1)(k-M)v_{M}.
$$
so $k=M$. By \cite{KMP}, finite decomposition at positive integral levels can happen only if $k=1$.
\end{proof}

It remains to deal with the cases with $\a_p$ short. Looking at tables in the previous section these are $C_{h}\times C_{n-h}\hookrightarrow C_n$, $D_n\hookrightarrow B_n$, $A_2\hookrightarrow G_2$, and $B_4\hookrightarrow F_4$. 
The latter three cases are dealt with in the work of Adamovi\'c and Per\v{s}e (see \cite[3.5]{KMP} �for a thorough discussion and precise attributions). The first case is analyzed in the next subsection.
\subsection{$C_{h}\times C_{n-h}$} Consider the level $-\tfrac{1}{2}$ case. Then Theorem \ref{general} applies and the decomposition is finite. \par
Let us deal with level $-1-\tfrac{n}{2}$. Realize as usual the root system of type $C_n$ in terms of the standard basis in $\mathbb R^n$ as
$\pm\{\e_i\pm\e_j,\mid i\ne j\}\cup\pm\{2\e_i\}$.
We can choose  root vectors in such a way that the following relations hold:
\begin{align*}
&[X_{-2\e_1},X_{\e_1+\e_{h+1}}]=-X_{-\e_1+\e_{h+1}}, &&[X_{-\e_1+\e_{h+1}},X_{\e_1+\e_{h+1}}]=2X_{2\e_{h+1}},\\
&[X_{-2\e_{h+1}},X_{\e_1+\e_{h+1}}]=-X_{\e_1-\e_{h+1}}, && [X_{\e_1-\e_{h+1}},X_{\e_1+\e_{h+1}}]=2X_{2\e_1},\\
& [X_{-\e_1+\e_{h+1}},X_{2\e_1}]=X_{\e_1+\e_{h+1}}, &&[X_{\e_1-\e_{h+1}},X_{2\e_{h+1}}]=X_{\e_1+\e_{h+1}}.
\end{align*}
Moreover
$$
(X_{2\e_i},X_{-2\e_i})=1,\ i=1,\ldots,n.
$$
With this choice of root vectors set $$v_{i,j}=(X_{\e_1+\e_{h+1}})^i_{(-1)}(X_{2\e_1})^j_{(-1)}(X_{2\e_{h+1}})^j_{(-1)}\vac.$$
\begin{lemma}\label{action2e1}
\begin{equation}\label{e1}(X_{-2\e_1})_{(1)}v_{i,j}=(k-i-j+1)j(X_{2\e_{h+1}})_{(-1)}v_{i,j-1}-i(i-1)(X_{2\e_{h+1}})_{(-1)}v_{i-2,j}.
\end{equation}
\begin{equation}(X_{-2\e_{h+1}})_{(1)}v_{i,j}=(k-i-j+1)j(X_{2\e_1})_{(-1)}v_{i,j-1}-i(i-1)(X_{2\e_1})_{(-1)}v_{i-2,j}.
\end{equation}
\begin{equation}\label{-e-e}(X_{-\e_1-\e_{h+1}})_{(1)}v_{i,j}=i(2k-(i-1)-4j)v_{i-1,j}-j^2v_{i+1,j-1}.
\end{equation}
\end{lemma}
\begin{proof}
Direct computation, by  induction  on $i+j$.
\end{proof}
\begin{lemma}Consider $v_{i,j}$ as an element of $V^k(\g)$ and let $S_m$  be the linear span of $\{v_{i,j}\mid i+2j=m\}$. 
Set
$$
v_m=\begin{cases}\sum\limits_{s=0}^t \frac{\binom{k-t+1}{s}\binom{t}{s}}{\binom{2s}{s}}v_{2s,t-s},\quad &\text{if $m=2t$,}\\
\sum\limits_{s=0}^t \frac{\binom{k-t}{s}\binom{t}{s}}{(s+1)\binom{2s+1}{s}}v_{2s+1,t-s}\quad &\text{if $m=2t+1$.}
\end{cases}
$$
Then
\begin{enumerate}
\item $\C v_m$ is the space of   $\widehat{\g^0}$-singular vector in $S_m$.
\item  
$
(X_{-\e_1-\e_{h+1}})_{(1)}v_m=c_mv_{m-1}
$
with
$$c_m=\begin{cases}
\frac{m}{2}(m^2+(-4-3k)m+2k^2+5k+3)&\text{if $m$ is even,}\\
2k-2m+2&\text{if $m$ is odd.}
\end{cases}
$$
\end{enumerate}
\end{lemma}
\begin{proof}Assume first $m=2t$ even.
In order to prove that $v_m$ is $\widehat{\g^0}$-singular we need only to check that
$$
(X_{-2\e_1})_{(1)}v_{m}=(X_{-2\e_{h+1}})_{(1)}v_{m}=0.
$$
Set  $a_{s,t}=\frac{\binom{k-t+1}{s}\binom{t}{s}}{\binom{2s}{s}}$. Then, by Lemma \ref{action2e1}, 
\begin{align}\notag
&(X_{-2\e_1})_{(1)}v_{m}
=\sum_{s=0}^{t-1} a_{s,t}(k-s-t+1)(t-s)(X_{2\e_{h+1}})_{(-1)}v_{2s,t-s-1}\\\notag
&-2\sum_{s=1}^t a_{s,t}s(2s-1)(X_{2\e_{h+1}})_{(-1)}v_{2s-2,t-s}=\\\label{ezero}
&(X_{2\e_{h+1}})_{(-1)}\sum_{s=0}^{t-1} (a_{s,t}(k-s-t+1)(t-s)-a_{s+1,t}2(s+1)(2s+1))v_{2s,t-s-1}.
\end{align}
Since
$\frac{(k-s-t+1)(t-s)}{2(s+1)(2s+1)}a_{s,t}=a_{s+1,t}, 
$
\eqref{ezero} Êvanishes, as required.
The same computation shows that $(X_{-2\e_{h+1}})_{(1)}v_{m}=0$.

Let $v=\sum_{s=0}^t c_{s,t}v_{2s,t-s}$ be $\widehat{\g^0}$-singular. Then the computation above and the fact that the $v_{i,j}$ are linearly independent show that
$$
(k-s-t+1)(t-s)c_{s,t}=2(s+1)(2s+1)c_{s+1,t},
$$
so $v$ is determined by the choice of $c_{0,t}$, thus  the space  of $\widehat{\g^0}$-singular vectors in $S_m$ is one-dimensional. This proves (1).
\par

Next observe that 
$(X_{-\e_1-\e_{h+1}})_{(1)}v_m$ is $\widehat{\g^0}$-singular. It suffices to prove  that
$$(X_{\e_1-\e_2})_{(0)}(X_{-\e_1-\e_{h+1}})_{(1)}v_m=(X_{\e_{h+1}-\e_{h+2}})_{(0)}(X_{-\e_1-\e_{h+1}})_{(1)}v_m=0.$$ Let us check only that  $(X_{\e_1-\e_2})_{(0)}(X_{-\e_1-\e_{h+1}})_{(1)}v_m=0$, since the other equality is obtained similarly. The latter follows at once from the fact that 
$$
(X_{\e_1-\e_2})_{(0)}(X_{-\e_1-\e_{h+1}})_{(1)}v_m=(X_{-\e_2-\e_{h+1}})_{(1)}v_m
$$
and that 
\begin{align*}
(X_{-\e_2-\e_{h+1}})&_{(1)}v_{i,j}=i(X_{\e_1+\e_{h+1}})^{i-1}_{(-1)}(X_{2\e_1})^j_{(-1)}(X_{2\e_{h+1}})^j_{(-1)}(X_{\e_1-\e_2})_{(0)}\vac\\&+j(X_{\e_1+\e_{h+1}})^{i}_{(-1)}(X_{2\e_1})^j_{(-1)}(X_{2\e_{h+1}})^{j-1}_{(-1)}(X_{-\e_2+\e_{h+1}})_{(0)}\vac=0.
\end{align*}
It follows from \eqref{-e-e} that $(X_{-\e_1-\e_{h+1}})_{(1)}v_m\in S_{m-1}$, so, since the space  of $\widehat{\g^0}$-singular vectors in $S_{m-1}$ is one-dimensional, $(X_{-\e_1-\e_{h+1}})_{(1)}v_m=c_m v_{m-1}$. To compute the coefficient $c_m$ we need only to compute the coefficient of $v_{1,m/2-1}$ in $(X_{-\e_1-\e_{h+1}})_{(1)}v_m$ if $m$ is even and the coefficient of $v_{0,(m-1)/2}$ in $(X_{-\e_1-\e_{h+1}})_{(1)}v_m$ if $m$ is odd. By \eqref{-e-e} this coefficient is $-(m/2)^2+(2k-2m+3)(k-m/2+1)m/2$ if $m$ is even and it is equal to $(2k-2(m-1))$ if $m$ is odd.
\end{proof}

\begin{cor}\label{46} In type $C_n$ with $n\ge 2$, if $k=-1-\tfrac{n}{2}$, then, for each $m\ge0$, $v_m$ projects to a nonzero $\widehat{\g^0}$-singular vector in $V_k(\g)$. In particular, the decomposition of $V_k(\g)$ as $\widehat{\g^0}$-module cannot be finite.
\end{cor}
\begin{proof}For the first statement, we need only to check that $c_m\ne 0$ if $m\ge1$. In fact, since $
(X_{-\e_1-\e_{h+1}})_{(1)}v_m=c_mv_{m-1}$, the result follows by an obvious induction.

If $k=-1-\tfrac{n}{2}$ and $m$ is odd, then $c_m=-n-2m<0$. If $m$ is even, then we need to check that
\begin{equation}\label{equaz}
m^2+(-1+\tfrac{3}{2}n)m+\tfrac{n(n-1)}{2}\ne 0.
\end{equation}
Solving for $m$ the equation $
m^2+(-1+\tfrac{3}{2}n)m+\tfrac{n(n-1)}{2}= 0
$, we find $m=1-n$ or $m=-\tfrac{n}{2}$, hence \eqref{equaz} holds.

Since the vectors $v_m$ have different weights, they are linearly independent, thus the second statement follows.
\end{proof}
\begin{prop}\label{cw} If $\g^0$ is a maximal semisimple equal rank subalgebra of $\g$ and  $k$ is a conformal level, then there are infinitely many $\widehat{\g^0}$-singular vectors in $V_k(\g)$ if and only if there is a $\widehat{\g^0}$-singular vector  in $V_k(\g)$ having  conformal weight  $2$.
\end{prop}
\begin{proof} Recall that 
$\g^0$ is the fixed-points subalgebra of an automorphism $\sigma$ of $\g$ of finite order. 
If the decomposition is finite then the singular vectors are $\vac$ and the vectors $x_{(-1)}\vac$ with $x$ an highest weight vector for the components $V(\mu_i)$ of $\p$. In fact,
if  $k=1$, then we know from \cite{Kac} that there are  $a_p$ summands ($a_p$ being the label of $\a_p$) in the $\widehat{\g^0}$-decomposition of $V_k(\g)$. If $x$ is an highest weight vector for $V(\mu_i)$, then $x_{(-1)}\vac$ is a singular vector for $\widehat{\g^0}$. Since $a_p$ coincides with the order of $\s$, these singular vectors give the whole decomposition. If $k\ne 1$, we apply Theorem \ref{general}, to obtain the same result.

If the decomposition is not finite, then it is easy to check that at least one of the infinitely many singular vectors that we constructed has conformal weight $2$.
\end{proof}

\begin{rem}
In Appendix \ref{6} we give the explicit decompositions for all the cases where finite decomposition occurs.
\end{rem}
\section{Finite decomposition for maximal equal rank reductive embeddings}\label{FDHS}
Assume that $\g^0$ is a maximal equal rank subalgebra of $\g$ such that the center $\g^0_0$ of $\g^0$ is nonzero. Let $\varpi$ and $\zeta$ be as in Section \ref{AeP}.
As shown in \cite{CKMP}, at level $k=1$ the decomposition is not finite, but the eigenspaces of the action of $\varpi_{(0)}$ admit finite decomposition. In this section we  discuss finite decomposition of $\varpi_{(0)}$-eigenspaces for conformal levels different from $1$. In the following we write weights of the simple ideals $\g^0_j$ of $\g^0$ as linear combinations of the fundamental weights $\omega_i$ of $\g^0_j$ (to avoid cumbersome notation, we will not make explicit the depencence on $j$ unless it is necessary).\par
In Theorem \ref{51} we list the cases where condition \eqref{finioths} holds, hence Theorem \ref{generalhs} applies. Case (5) already appears in \cite{AP2}; case (4) can be derived  at once from \cite{AP2} using results from  \cite{FF}; the methods of \cite{AP2} actually cover also case (4) with $n=3$,  where \eqref{finioths} does not hold: see Theorem \ref{nequal5} below, where 
more instances of finite decomposition will be given. 
\begin{theorem}\label{51} Assume we are in one of the following cases.
\begin{enumerate}
\item Type $A_{h-1}\times A_{n-h}\times \C\varpi$ in $A_{n}$ with $n>5$, $h>2$ and $n-h>1$, conformal level $k=-1$.
\item Type $A_{n-1}\times \C\varpi$ in $A_{n}$ with $n>3$, conformal level $k=-\frac{n+1}{2}$.
\item Type $A_{n-1}\times \C\varpi$ in $D_{n}$ with $n>4$, conformal level $k=-2$.
\item Type $A_{n-1}\times \C\varpi$ in $C_{n}$ with $n> 3$, conformal level $k=-\half$.
\item Type $D_{5}\times \C\varpi$ in $E_{6}$, conformal level $k=-3$.
\item Type $E_{6}\times \C\varpi$ in $E_{7}$, conformal level $k=-4$.
\end{enumerate}
Then $\widetilde{V}_k(\g^0)\cong V_{\mathbf{k}}(\g^0)$, the $\varpi_{(0)}$-eigenspaces in $V_k(\g)$ are irreducible $\widehat{\g^0}$-modules and the decomposition of $V_k(\g)$ as a $\widehat{\g^0}$-module  is given by formula \eqref{efd}.
\end{theorem}
\begin{proof}As observed in Remark \ref{conditionsfinite}, it is enough to check in each case that condition \eqref{finioths} holds:
\begin{enumerate}
\item In this case $V(\theta)=V_{A_{h-1}}(\omega_1)\otimes V_{A_{n-h}}(\omega_{n-h})\otimes V_{\mathbb C \varpi}(\zeta)$ and $V(-\a_h)=V_{A_{h-1}}(\omega_{h-1})\otimes V_{A_{n-h}}(\omega_{1})\otimes V_{\mathbb C \varpi}(-\zeta)$, thus $V(\theta)\otimes V(-\a_h)$ decomposes as
\begin{align*}
&V_{A_{h-1}}(0)\otimes V_{A_{n-h}}(0)\otimes V_{\mathbb C \varpi}(0)\\&\oplus V_{A_{h-1}}(\omega_1+\omega_{h-1})\otimes V_{A_{n-h}}(0)\otimes V_{\mathbb C \varpi}(0)\\&\oplus V_{A_{h-1}}(0)\otimes V_{A_{n-h}}(\omega_1+\omega_{n-h})\otimes V_{\mathbb C \varpi}(0)\\
&\oplus  V_{A_{h-1}}(\omega_1+\omega_{h-1})\otimes V_{A_{n-h}}(\omega_1+\omega_{n-h})\otimes V_{\mathbb C \varpi}(0),
\end{align*}
so the corresponding formula in \eqref{finioths} yields for the nontrivial summands above
$1+\frac{1}{h-1}$, $1+\frac{1}{n-h}$, $2+\frac{1}{h-1}+\frac{1}{n-h}$ respectively. 
\item  We have $V(\theta)=V_{A_{n-1}}(\omega_{n-1})\otimes V_{\mathbb C \varpi}(\zeta)$ and $V(-\a_1)=V_{A_{n-1}}(\omega_{1})\otimes V_{\mathbb C \varpi}(-\zeta)$, thus $V(\theta)\otimes V(-\a_1)$ decomposes as
\begin{align*}
V_{A_{n-1}}(0)\oplus V(0)\otimes V_{A_{n-1}}(\omega_1+\omega_{n-1})\otimes V_{\mathbb C \varpi}(0),
\end{align*}
so the corresponding formula in \eqref{finioths} yields for the nontrivial summand above
$2+\frac{2}{n-1}$. 
\item  We have $V(\theta)=V_{A_{n-1}}(\omega_{2})\otimes V_{\mathbb C \varpi}(\zeta)$ and $V(-\a_n)=V_{A_{n-1}}(\omega_{n-2})\otimes V_{\mathbb C \varpi}(-\zeta)$, thus $V(\theta)\otimes V(-\a_n)$ decomposes as
\begin{align*}
&V_{A_{n-1}}(0)\otimes V_{\mathbb C \varpi}(0)\\
&\oplus V_{A_{n-1}}(\omega_1+\omega_{n-1})\otimes V_{\mathbb C \varpi}(0)\\
&\oplus V_{A_{n-1}}(\omega_2+\omega_{n-2})\otimes V_{\mathbb C \varpi}(0),
\end{align*}
so the corresponding formula in \eqref{finioths} yields for the nontrivial summands above
$1+\frac{2}{n-2}$, $2+\frac{2}{n-2}$, respectively.
\item  We have $V(\theta)= V_{A_{n-1}}(2\omega_{1})\otimes V_{\mathbb C \varpi}(\zeta)$ and $V(-\a_n)= V_{A_{n-1}}(2\omega_{n-1})$, thus $V(\theta)\otimes V(-\a_n)\otimes V_{\mathbb C \varpi}(-\zeta)$ decomposes as
\begin{align*}
&V_{A_{n-1}}(0)\otimes V_{\mathbb C \varpi}(0)\\
&V_{A_{n-1}}(\omega_1+\omega_{n-1})\otimes V_{\mathbb C \varpi}(0)\\
&V_{A_{n-1}}(2\omega_1+2\omega_{n-1})\otimes V_{\mathbb C \varpi}(0),
\end{align*}
so the corresponding formula in \eqref{finioths} yields for the nontrivial summands above
$1+\frac{1}{n-1}$, $2+\frac{2}{n-1}$, respectively.
\item  We have  $V(\theta)=V_{D_{5}}(\omega_{4})\otimes V_{\mathbb C \varpi}(\zeta)$, $V(-\a_1)=V_{D_{5}}(\omega_{5})\otimes V_{\mathbb C \varpi}(-\zeta)$, thus $V(\theta)\otimes V(-\a_1)$ decomposes as
\begin{align*}
&V_{D_{5}}(0)\otimes V_{\mathbb C \varpi}(0)\\
&V_{D_{5}}(\omega_2)\otimes V_{\mathbb C \varpi}(0)\\
&V_{D_{5}}(\omega_4+\omega_{5})\otimes V_{\mathbb C \varpi}(0),
\end{align*}
so the corresponding formula in \eqref{finioths} yields for the nontrivial summands above
$\frac{8}{5}$, $\frac{12}{5}$, respectively.
\item  We have  $V(\theta)=V_{E_{6}}(\omega_{1})\otimes V_{\mathbb C \varpi}(\zeta)$, $V(-\a_7)=V_{E_{6}}(\omega_{6})\otimes V_{\mathbb C \varpi}(-\zeta)$, thus $V(\theta)\otimes V(-\a_7)$ decomposes as
\begin{align*}
&V_{E_{6}}(0)\otimes V_{\mathbb C \varpi}(0)\\
&\oplus V_{E_{6}}(\omega_2)\otimes V_{\mathbb C \varpi}(0)\\
&\oplus V_{E_{6}}(\omega_1+\omega_{6})\otimes V_{\mathbb C \varpi}(0),
\end{align*}
so the corresponding formula in \eqref{finioths} yields for the nontrivial summands above
$\frac{3}{2}$, $\frac{9}{4}$, respectively.
\end{enumerate}
\end{proof}

The next result discusses two special cases where  we still have finite decomposition of the eigenspaces of $\varpi_{(0)}$, even though the criterion in \eqref{finioths} does not apply.
\begin{theorem}\label{nequal5}
Consider the embeddings
\begin{enumerate}
\item  $A_{2}\times A_{2}\times \C\varpi\subset A_5$ at conformal level  $k=-1$;
\item $A_{2}\times  \C\varpi$ in $C_{3}$ with conformal level $k=-1/2$.
\end{enumerate}
Then $\widetilde V_k(\g^0)\cong V_\mathbf{k}(\g^0)$ and as a $V_\mathbf{k}(\g^0)$-module,
\begin{equation}
\label{eigenweightt}
V_k(\g)^{(q)}=\begin{cases}L_{\g^0}(q\theta)&\text{if $q\ge 0$,}\\
L_{\g^0}(q\a_p)&\text{if $q\le 0$.}
\end{cases}
\end{equation}
\end{theorem}
\begin{proof} Case (1).
Recall that $V_{(-1,-1)}(gl(3))=V_{-1}(A_2)\otimes V_{-1}(\C I_3)$ and let $\zeta_{I_3}\in(\C I_3)^*$ be defined by $\zeta_{I_3}(I_3)=1$ ($I_n$ being the identity matrix in $gl(n)$). If $x\in gl(3)$ we set $\dot x=(x,0)\in gl(3)\times gl(3)$ and $\ddot x=(0,x)\in gl(3)\times gl(3)$. Likewise we write $\l\in (gl(3)\times gl(3))^*$ as $\l=\dot\l+\ddot\l$ with $\dot\l(\ddot x)=\ddot\l(\dot x)=0$ for all $x\in gl(3)$. Set also $\mathfrak c=span(\dot I_3,\ddot I_3)$ be the center of $gl(3)\times gl(3)$.

We use the Kac-Wakimoto free field realization of $V_{(-1,-1)}(gl(n+1))$ (see \cite{Kw}). As shown in \cite{AP2},  $V_{(-1,-1)}(gl(n+1))$ can be realized for $n>1$ as the subspace $F^0$ of zero total charge of the universal vertex algebra $F$ generated by fields $a^\pm_i$ ($i=1,\ldots,n+1$) with $\l$-products
$$
[(a_i^+)_\l (a_j^-)]=\d_{ij},\quad[(a_i^-)_\l (a_j^+)]=-\d_{ij},\quad[(a_i^+)_\l (a_j^+)]=[(a_i^-)_\l (a_j^-)]=0.
$$

Restrict now to the $n=5$ case.
Let $F_1$ be the vertex subalgebra generated by $a^\pm_i$ with $i\le 3$ and $F_{2}$ the subalgebra generated by $a^\pm_i$ with $i> 3$. Then, by the explicit realization of $F$ given in \cite{AP2}, we have that, as a $V_{(-1,-1)}(gl(3))\otimes V_{(-1,-1)}(gl(3))$-module, 
$$
F=F_1\otimes F_{2},
$$
hence
$$F^0=\sum_{q\in\ganz} F^q_1\otimes F^{-q}_{2},
$$
where we define $F_j^q$ to be the subspace of total charge $q$. 
\par
We can now apply Theorem 3.2 of \cite{AP2} to $F_j^q$ and find that
$$
F_j^q=\begin{cases}L_{A_2}(q\omega_1)\otimes L_{\mathbb C I_3}(q\zeta_{I_3})&\text{if $q\ge 0$,}\\L_{A_2}(-q\omega_{2})\otimes L_{\mathbb C I_3}(q\zeta_{I_3})&\text{if $q< 0$.}
\end{cases}
$$
as a $V_{(-1,-1)}(gl(3))$-module.

Thus,
as a $V_{(-1,-1)}(A_{2}\times A_{2})\otimes V_{(-1,-1)}(\mathfrak c)$-module,
$$
F^q_1\otimes F^{-q}_{2}=\begin{cases}L_{A_2\times A_2}(q(\dot \omega_1+\ddot\omega_{2}))\otimes L_{\mathfrak  c}(q(\dot\zeta_{\dot I_3}-\ddot\zeta_{\ddot I_{3}}))&\text{if $q\ge 0$,}\\L_{A_2\times A_2}(-q(\dot\omega_2+\ddot\omega_{1}))\otimes L_{\mathfrak  c}(q(\dot\zeta_{\dot I_3}-\ddot\zeta_{\ddot I_{3}}))&\text{if $q< 0$.}
\end{cases}
$$

To recover the action of $V_{-1}(A_{2})\otimes V_{-1}(A_{2})\otimes V_{-1}(\C\varpi)$ on $F^q_1\otimes F^{-q}_{2}$  we observe that both $\{\dot I_3,\ddot I_{3}\}$ and $\{I_{6},\varpi\}$ are orthogonal bases of the center of $gl(3)\times gl(3)$.  Since $I_{6}=\dot I_3+\ddot I_{3}$ while $\varpi=\tfrac{1}{2}\dot I_3-\tfrac{1}{2}\ddot I_{3}$, it follows that $\dot\zeta_{\dot I_3}=\tfrac{1}{2}\zeta+\zeta_{I_{6}}$ while $\ddot\zeta_{\ddot I_{3}}=-\tfrac{1}{2}\zeta+\zeta_{I_{6}}$. Thus, as a $V_{(-1,-1)}(\mathfrak c)=V_{-1}(\C \varpi)\otimes V_{-1}(\C I_{6})$-module,
$$
L_{\mathfrak  c}(q(\dot\zeta_{\dot I_3}-\ddot\zeta_{\ddot I_{3}}))=L_{\C \varpi}(q\zeta)\otimes L_{\C I_6}(0).
$$

Let $(F^0)^+=\{v\in F^0\mid (I_{6})_{(r)}v=0\text{ for all } r>0\}$. By Theorem 3.2 of \cite{AP2}, $(F^0)^+=V_{-1}(A_5)$, so we obtain that, if $q\ge 0$, then
\begin{align*}
V_{-1}(A_5)^{(q)}&=\left(L_{A_2\times A_2}(q(\dot \omega_1+\ddot\omega_{2}))\otimes L(q\zeta)\otimes L(0)\right)^+\\
&=L_{A_2\times A_2}(q(\dot \omega_1+\ddot\omega_{2}))\otimes L(q\zeta),
\end{align*}
and, if $q\le 0$,
\begin{align*}
V_{-1}(A_5)^{(q)}&=\left(L_{A_2\times A_2}(-q(\dot \omega_2+\ddot\omega_{1}))\otimes L(q\zeta)\otimes L(0)\right)^+\\
&=L_{A_2\times A_2}(-q(\dot \omega_2+\ddot\omega_{1}))\otimes L(q\zeta).
\end{align*}
This in particular shows that $V_{-1}(A_5)^{(0)}$ is  simple as a $\widehat{\g^0}$-module and, since it contains $\widetilde V_{-1}(\g^0)$ as a submodule, we have that $\widetilde V_{-1}(\g^0)$ is simple. 

To finish the proof it is enough to observe that
$$
\dot \omega_1+\ddot\omega_{2}+\zeta=\theta,\ -\dot \omega_2-\ddot\omega_{1}+\zeta=\a_p.
$$

Case (2). Again by  \cite[Theorem 3.2 ]{AP2}  we have 
\begin{equation}\label{eqc}
F^q=\begin{cases}L_{A_2}(q\omega_1)\otimes L_{\mathbb C I_3}(q\zeta_{I_3})&\text{if $q\ge 0$}\\L_{A_2}(-q\omega_{2})\otimes L_{\mathbb C I_3}(q\zeta_{I_3})&\text{if $q< 0$}
\end{cases}
\end{equation}
as a $V_{(-1,-1)}(gl(3))$-module.   
 But by \cite[Table XII]{FF}, the  even part of $F$ is the  simple vertex algebra $V_{-1/2}(C_n)$. This immediately gives the decomposition  
\begin{equation}\label{eqcc}V_{-1/2}(C_n) = \bigoplus_{q \in \mathbb Z} F^{2q}.\end{equation}
Since $2\zeta_{I_3}+2\omega_1=\zeta+2\omega_1=\theta$ and $2\zeta_{I_3}-2\omega_2=\zeta-2\omega_2 =\a_p$, combining \eqref{eqc} and \eqref{eqcc}, formula \eqref{eigenweightt} follows in this case too. \end{proof}

In the next result we collect a few cases where finite decomposition does not occur.

\begin{theorem} \label{infinidec} In the following cases the decomposition of the $q$-eigenspace for $\varpi_{(0)}$ in $V_k(\g)$ as $\widehat{\g^0}$-module is not finite for any $q\in\ganz$:
\begin{enumerate}
\item Type $A_{h-1}\times A_{n-h}\times \C\varpi$ in $A_{n}$ with $h\ge 2$ and $n-h\ge1$, conformal level  $k=-\tfrac{n+1}{2}$.
\item Type $A_{h-1}\times A_{n-h}\times \C\varpi$ in $A_{n}$ with either $h= 2$ or  $n-h=1$, conformal level  $k=-1$.
\item Type $A_{1}\times \C\varpi$ in $A_{2}$, conformal level  $k=-\frac{3}{2}$.
\item Type $A_1 \times {\mathbb C}  \varpi  $ in $B_2 =C_2$, conformal level $k =-1/2$.
\end{enumerate}
\end{theorem}

The proof of the above result employs  very different techniques in each of the four cases. We will discuss these cases   in the following  subsections. 
\begin{rem}\label{missing}
The remaining open cases are the following (we set $D_3=A_3$).
\begin{enumerate}
\item Type $A_{2}\times \C\varpi$ in $A_{3}$, conformal level  $k=-2$.
\item Type $D_{n-1}\times \C\varpi$ in $D_{n}$, conformal level  $k=2-n$, $n\geq 4$.
\item Type $B_{n-1}\times \C\varpi$ in $B_{n}$, conformal level  $k=\frac{3}{2}-n$, $n\geq 3$.
\end{enumerate}
\end{rem}

\subsection{Proof of Theorem \ref{infinidec} (1)}
We can choose the root vectors in such a way that the following relations hold:
$$
[X_{\e_i-\e_j},X_{\e_r-\e_s}]=\delta_{j,r}X_{\e_i-\e_s}-\delta_{s,i}X_{\e_r-\e_j}.
$$
Moreover
$(X_{\a},X_{-\a})=1$ for all roots $\a$. With this choice of roots vectors set 
$$v_{i,j,r,s}=(X_{-\e_{h}+\e_{h+1}})^i_{(-1)}(X_{\e_{1}-\e_{n+1}})_{(-1)}^j(X_{\e_1-\e_h})^r_{(-1)}(X_{\e_{h+1}-\e_{n+1}})_{(-1)}^s\vac.
$$
Note that $\varpi_{(0)} v_{i,j,r,s}=(j-i)v_{i,j,r,s}$ and that if $\a$ is a positive root for $\g^0$ then $(X_\a)_{(0)} v_{i,j,r,s}=0$.
\begin{lemma}\label{actioneij}
\begin{equation}\label{ea1}(X_{-\e_1+\e_h})_{(1)}v_{i,j,r,s}=r(k-i-j-r+1)v_{i,j,r-1,s}-i j\, v_{i-1,j-1,r,s+1}.
\end{equation}
\begin{equation}(X_{-\e_{h+1}+\e_{n+1}})_{(1)}v_{i,j,r,s}=s(k-i-j-s+1)v_{i,j,r,s-1}-i j\, v_{i-1,j-1,r+1,s}.
\end{equation}
\begin{equation}\label{-e-ea}(X_{\e_h-\e_{h+1}})_{(1)}v_{i,j,r,s}=i(k-i-r-s+1)v_{i-1,j,r,s}-r s\, v_{i,j+1,r-1,s-1}.
\end{equation}
\begin{equation}\label{-ewhatever}(X_{-\e_1+\e_{n+1}})_{(1)}v_{i,j,r,s}=j(k-j-r-s+1)v_{i,j-1,r,s}-r s\, v_{i+1,j,r-1,s-1}.
\end{equation}
\begin{equation}\label{-ecenter}(\varpi)_{(t)}v_{i,j,r,s}=0,\ t>0.\end{equation}
\end{lemma}
\begin{proof}
The induction is on $i+j+r+s$ with the base case $i=j=r=s=0$  being clear. We give the details only for \eqref{ea1}. If $i=j=r=0$ it is clear that $(X_{-\e_1+\e_h})_{(1)}v_{0,0,0,s}=0$; if $i=j=0, r>0,$ then
\begin{align*}
&(X_{-\e_1+\e_h})_{(1)}v_{0,0,r,s}\\&=(-(h_{-\e_1+\e_h})_{(0)}+k)v_{0,0,r-1,s}+(X_{\e_1-\e_h})_{(-1)}(X_{-\e_1+\e_h})_{(1)}v_{0,0,r-1,s}\\
&=(-2(r-1)+k)v_{0,0,r-1,s}+(r-1)(k-r+2)v_{0,0,r-1,s}\\
&=r(k-r+1)v_{0,0,r-1,s}.\end{align*}
If $i=0, j>0$
\begin{align}
&(X_{-\e_1+\e_h})_{(1)}v_{0,j,r,s}=\notag\\&(X_{\e_{h}-\e_{n+1}})_{(0)}v_{0,j-1,r,s}+(X_{\e_{1}-\e_{n+1}})_{(-1)}(X_{-\e_1+\e_h})_{(1)}v_{0,j-1,r,s}=\notag\\
&(X_{\e_{h}-\e_{n+1}})_{(0)}v_{0,j-1,r,s}+r(k-j-r+2)v_{0,j,r-1,s}.\label{ehn1}\end{align}

We claim that 
$$
(X_{\e_{h}-\e_{n+1}})_{(0)}v_{0,j,r,s}=-rv_{0,j+1,r-1,s}.
$$
Since $[(X_{\e_{h}-\e_{n+1}})_{(0)},(X_{\e_{1}-\e_{n+1}})_{(-1)}]=0$, we need only to prove that 
$$
(X_{\e_{h}-\e_{n+1}})_{(0)}v_{0,0,r,s}=-rv_{0,1,r-1,s}.
$$
We use induction on $r$, the base $r=0$ being clear.
If $r>0$, then
\begin{align*}
&(X_{\e_{h}-\e_{n+1}})_{(0)}v_{0,0,r,s}\\&=-(X_{\e_{1}-\e_{n+1}})_{(-1)}v_{0,0,r-1,s}+(X_{\e_{1}-\e_{h}})_{(-1)}(X_{\e_{h}-\e_{n+1}})_{(0)}v_{0,0,r-1,s}\\
&=-v_{0,1,r-1,s}-(r-1)v_{0,1,r-1,s}=-rv_{0,1,r-1,s}
\end{align*}
Substituting in \eqref{ehn1} we find
\begin{align*}
(X_{-\e_1+\e_h})_{(1)}v_{0,j,r,s}&=-rv_{0,j,r-1,s}+r(k-j-r+2)v_{0,j,r-1,s}\\
&=r(k-j-r+1)v_{0,j,r-1,s}.
\end{align*}
Finally, if $i>0$,
\begin{align}
&(X_{-\e_1+\e_h})_{(1)}v_{i,j,r,s}\notag\\&=-(X_{-\e_{1}+\e_{h+1}})_{(0)}v_{i-1,j,r,s}+(X_{-\e_{h}+\e_{h+1}})_{(-1)}(X_{-\e_{1}+\e_{h}})_{(1)}v_{i-1,j,r,s}\notag\\
&=-(X_{-\e_{1}+\e_{h+1}})_{(0)}v_{i-1,j,r,s}\notag\\
&+r(k-i-j-r+2)v_{i,j,r-1,s}-(i-1)jv_{i-1,j-1,r,s+1}\label{e1h}.
\end{align}
We claim that
$$
(X_{-\e_{1}+\e_{h+1}})_{(0)}v_{i,j,r,s}=rv_{i+1,j,r-1,s}+jv_{i,j-1,r,s+1}.
$$
Since $[(X_{-\e_{1}+\e_{h+1}})_{(0)},(X_{-\e_{h}+\e_{h+1}})_{(-1)}]=0$, we need only to prove that 
$
(X_{-\e_{1}+\e_{h+1}})_{(0)}v_{0,j,r,s}=rv_{1,j,r-1,s}+jv_{0,j-1,r,s+1}
$.
The induction is on $j+r$. The base $j=r=0$ is clear. If $j=0$ and $r>0$,
\begin{align*}
&(X_{-\e_{1}+\e_{h+1}})_{(0)}v_{0,0,r,s}\\&=(X_{-\e_{h}+\e_{h+1}})_{(-1)}v_{0,0,r-1,s}+(X_{\e_{1}-\e_{h}})_{(-1)}(X_{-\e_{1}+\e_{h+1}})_{(0)}v_{0,0,r-1,s}\\
&=v_{1,0,r-1,s}+(r-1)v_{1,0,r-1,s}=rv_{1,0,r-1,s}.
\end{align*}
If $j>0$,
\begin{align*}
&(X_{-\e_{1}+\e_{h+1}})_{(0)}v_{0,j,r,s}\\&=(X_{\e_{h+1}-\e_{n+1}})_{(-1)}v_{0,j-1,r,s}+(X_{\e_{1}-\e_{n+1}})_{(-1)}(X_{-\e_{1}+\e_{h+1}})_{(0)}v_{0,j-1,r,s}\\
&=v_{0,j-1,r,s+1}+rv_{1,j,r-1,s}+(j-1)v_{0,j-1,r,s+1}=rv_{1,j,r-1,s}+jv_{0,j-1,r,s+1}.
\end{align*}
Substituting in \eqref{e1h}, we find
\begin{align*}
(X_{-\e_1+\e_h})_{(1)}v_{i,j,r,s}&=-rv_{i,j,r-1,s}-jv_{i-1,j-1,r,s+1}\\
&+r(k-i-j-r+2)v_{i,j,r-1,s}-(i-1)jv_{i-1,j-1,r,s+1}\\
&=r(k-i-j-r+1)v_{i,j,r-1,s}-ijv_{i-1,j-1,r,s+1},
\end{align*}
which is \eqref{ea1}.
\end{proof}

\begin{lemma}\label{55} Consider $v_{i,j,r,s}$ as an element of $V^k(\g)$. Fix $m\ge 0$.  Set, for $i=0,\ldots,m$,
$$w_{i,q}=\begin{cases}v_{i,i+q,m-i,m-i}&\text{if $q\ge 0$,}\\ 
v_{i-q,i,m-i,m-i}&\text{if $q\le 0$.}
\end{cases}
$$
Set 
$$S_{m,q}=span(w_{i,q}\mid i=0,\ldots,m)$$
and
$$
v_{m,q}=\sum_{i=0}^m\frac{ \binom{k-m-|q|+1}{i}\binom{m}{i}}{\binom{i+|q|}{|q|}}w_{i,q}.
$$
Then
\begin{enumerate}
\item $\C v_{m,q}$ is the space of   $\widehat{\g^0}$-singular vectors in $S_{m,q}$.
\item  If $q\ge0$ and $m>0$ then
$
(X_{\e_h-\e_{h+1}})_{(1)}v_{m,q}=c_{m,q}v_{m-1,q+1}
$
with
\begin{equation}\label{cmq}
c_{m,q}=\frac{m}{q+1}(2 + k - m) (1 + k - 2 m - q).
\end{equation}
\item  If $q\le0$ and $m>0$ then
$
(X_{-\e_1+\e_{n+1}})_{(1)}v_{m,q}=d_{m,q}v_{m-1,q-1}
$
with
\begin{equation}\label{cmq2}
d_{m,q}=\frac{m}{-q+1}(2 + k - m) (1 + k - 2 m + q).
\end{equation}
\end{enumerate}
\end{lemma}
\begin{proof}
By \eqref{-ecenter}, in order to prove that $v_{m,q}$ is $\widehat{\g^0}$-singular, we need only to check that
$$
(X_{-\e_1+\e_h})_{(1)}v_{m,q}=(X_{-\e_{h+1}+\e_{n+1}})_{(1)}v_{m,q}=0.
$$
Set $a_i=\frac{ \binom{k-m-|q|+1}{i}\binom{m}{i}}{\binom{i+|q|}{|q|}}$ and $$w'_{i,q}=\begin{cases}v_{i,i+q,m-i-1,m-i}&\text{if $q\ge 0$,}\\v_{i-q,i,m-i-1,m-i}&\text{if $q\le 0$.}
\end{cases}$$
 By Lemma \ref{actioneij}, 
\begin{align*}
&(X_{-\e_1+\e_h})_{(1)}v_{m,q}
=\sum_{i=0}^{m-1} a_{i}(k-i-m-|q|+1)(m-i)w'_{i,q}-\\
&\sum_{i=1}^{m} a_{i}i(i+|q|)w'_{i-1,q})=\\
&\sum_{i=0}^{m-1} (a_{i}(k-i-m-|q|+1)(m-i)-a_{i+1}(i+1)(i+|q|+1))w'_{i,q}=0.
\end{align*}
The last equality follows from the following relation
$$\frac{(k-i-m-|q|+1)(m-i)}{(i+1)(i+|q|+1)}a_{i}=a_{i+1}.
$$
The same computation shows that $(X_{-\e_{h+1}+\e_{n+1}})_{(1)}v_{m,q}=0$.

Let $v=\sum_{s=0}^m c_{i}w_{i}$ be $\widehat{\g^0}$-singular. Then the computation above and the fact that the $v_{i,j,r,s}$ are linearly independent show that
$$
(k-i-m-|q|+1)(m-i)c_{i}=(i+1)(i+|q|+1)c_{i+1},
$$
so $v$ is determined by the choice of $c_{0}$, thus the set of $\widehat{\g^0}$-singular vectors in $S_{m,q}$ is one-dimensional.

We now prove (2). Set $u=(X_{\e_h-\e_{h+1}})_{(1)}v_{m,q}$. We first observe that 
$u$ is $\widehat{\g^0}$-singular. Indeed, this check reduces to showing that 
$$
(X_{\e_{h-1}-\e_h})_{(0)}u=(X_{\e_{h+1}-\e_{h+2}})_{(0)}u=0.$$
Let us check only that  $(X_{\e_{h-1}-\e_h})_{(0)}u=0$;  the other equality is obtained similarly. 
Now
\begin{align*}
(X_{\e_{h-1}-\e_h})_{(0)}(X_{\e_h-\e_{h+1}})_{(1)}v_{m,q}&=(X_{\e_{h-1}-\e_{h+1}})_{(1)}v_{m,q}.
\end{align*}
We claim that $(X_{\e_{h-1}-\e_{h+1}})_{(1)}v_{i,j,r,s}=0$ for all $i,j,r,s$. Indeed, since 
$$
[(X_{\e_{h-1}-\e_{h+1}})_{(1)},(X_{\e_1-\e_{n+1}})_{(-1)}]=[(X_{\e_{h-1}-\e_{h+1}})_{(1)},(X_{\e_1-\e_{h}})_{(-1)}]=0
$$
we need only to check that 
$$
(X_{\e_{h-1}-\e_{h+1}})_{(1)}v_{i,0,0,s}=0.
$$
We prove this by induction on $i+s$. If $i=0$ then 
\begin{align*}
(X_{\e_{h-1}-\e_{h+1}})_{(1)}v_{0,0,0,s}&=(X_{\e_{h-1}-\e_{n+1}})_{(0)}v_{0,0,0,s-1}=0
\end{align*}
If $i>0$, then
 \begin{align*}
(X_{\e_{h-1}-\e_{h+1}})_{(1)}v_{i,0,0,s}&=(X_{\e_{h-1}-\e_{h}})_{(0)}v_{i-1,0,0,s}=0.
\end{align*}
It follows from \eqref{-e-ea} that $(X_{\e_h-\e_{h+1}})_{(1)}v_{m,q}\in S_{m-1,q+1}$, so, since the space of $\widehat{\g^0}$-singular vectors in $S_{m-1,q+1}$ is one-dimensional, $(X_{\e_h-\e_{h+1}})_{(1)}v_{m,q}=c_{m,q} v_{m-1,q+1}$. To compute the coefficient $c_{m,q}$ we need only to compute the coefficient of $v_{0,q+1,m-1,m-1}$ in $(X_{\e_h-\e_{h+1}})_{(1)}v_{m,q}$. By \eqref{-e-ea} this coefficient is 
$$
c_{m,q}=-m^2 +\tfrac{m}{q+1}(k-2m+2) (k-m-q+1))=\tfrac{m}{q+1}(2 + k - m) (1 + k - 2 m - q).
$$

We now prove (3). Set $u=(X_{-\e_1+\e_{n+1}})_{(1)}v_{m,q}$. By the same argument used in the proof of (2) one can check that   
$u$ is $\widehat{\g^0}$-singular. It follows from \eqref{-ewhatever} that $(X_{-\e_1+\e_{n+1}})_{(1)}v_{m,q}\in S_{m-1,q-1}$, so, since the space of $\widehat{\g^0}$-singular vectors in $S_{m-1,q-1}$ is one-dimensional, $(X_{-\e_1+\e_{n+1}})_{(1)}v_{m,q}=d_{m,q} v_{m-1,q-1}$. To compute the coefficient $d_{m,q}$ we need only to compute the coefficient of $v_{-q+1,0,m-1,m-1}$ in $(X_{-\e_1+\e_{n+1}})_{(1)}v_{m,q}$. By \eqref{-ewhatever} this coefficient is 
$$
d_{m,q}=\tfrac{m}{|q|+1}(2 + k - m) (1 + k - 2 m - |q|).
$$
\end{proof}

\begin{cor}Consider the embedding $A_{h-1}\times A_{n-h}\times \C\varpi\subset A_n$ with $h\ge 2$ and $n-h\ge1$. If $k=-\tfrac{n+1}{2}$, then, for each $m\ge0$ and $q\in \ganz$, $v_{m,q}$ projects to a nonzero $\widehat{\g^0}$-singular vector in $V_k(\g)$. In particular, the decomposition of the $q$-eigenspace for $\varpi_{(0)}$ in $V_k(\g)$ as $\widehat{\g^0}$-module cannot be finite.
\end{cor}
\begin{proof} Since $k=-\frac{n+1}{2}$ and $n\geq 3$, it is clear from \eqref{cmq} and \eqref{cmq2} that $c_{m,q}\ne 0$  for all $m\ge1$. An obvious induction using Lemma \ref{55} (2)  shows that, if $q\ge 0$ and $v_{0,q+m}\ne 0$, then $v_{m,q}\ne 0$. Likewise, if $q\le 0$ and $v_{0,q-m}\ne 0$ then $v_{m,q}\ne 0$. We need only to prove that $v_{0,q}\ne 0$ for all $q$. 

If $q\ge0$, from \eqref{-ewhatever} we deduce that, 
\begin{align*}
(X_{-\e_1+\e_{n+1}})_{(1)}v_{0,q}&=(X_{-\e_1+\e_{n+1}})_{(1)}v_{0,q,0,0}\\
&=q(k-q+1)v_{0,q-1,0,0}=q(k-q+1)v_{0,q-1}.
\end{align*}
 Since $q(k-q+1)\ne 0$ for all $q\ge 1$ an obvious induction shows that $v_{0,q}\ne0$.

Similarly, if $q\le 0$, from \eqref{-e-ea} we deduce that, 
\begin{align*}
(X_{\e_h-\e_{h+1}})_{(1)}v_{0,q}&=(X_{\e_h-\e_{h+1}})_{(1)}v_{-q,0,0,0}\\
&=-q(k+q+1)v_{-q-1,0,0,0}=-q(k+q+1)v_{0,q+1}.
\end{align*}
 Since $-q(k+q+1)\ne 0$ for all $q\le -1$ an obvious induction shows that $v_{0,q}\ne0$. 

The set  $\{v_{m,q}\mid m\ge1\}$ is linearly independent since the vectors $v_{m,q}$ have different weights, thus the second statement follows.
\end{proof}

\subsection{Proof of Theorem \ref{infinidec} (2)}
We now discuss the embedding of type $A_{h-1}\times A_{n-h}\times \C\varpi$ in $A_{n}$ with either $h= 2$ or  $n-h=1$, conformal level $k=-1$. Without loss of generality we may assume $h=2$.

We use the Kac-Wakimoto free field realization of $V_{-1}(gl(n+1))$ described in the proof of Theorem \ref{nequal5}.
With the notation used there, we have, as 
$V_{-1}(gl(2))\otimes V_{-1}(gl(n-1))$-modules, 
$$V_{(-1,-1)}(gl(n+1))=\sum_{q\in\ganz} F^q_1\otimes F^{-q}_{2}.
$$
We may deduce from \cite[Remark 3.3]{Kw} that 
\begin{equation}\label{decompA1}
F^q_1=\sum_{j=0}^\infty L_{A_1}((2j+|q|)\omega_1)\otimes L_{\C I_2}(q\zeta_{I_2})
\end{equation}
as $V_{-1}(A_1)\otimes V_{-1}(\C I_2)$-module.

Arguing as in the proof of Theorem \ref{nequal5} we see that, if $n>3$,
$$
V_{-1}(A_n)^{(q)}=\sum_{j\ge0}L_{A_1}((2j+q)\omega_1)\otimes L_{A_{n-2}}(q\omega_{n-2})\otimes L_{\C\varpi}(q\zeta)
$$
when $q\ge0$, while, for $q\le 0$,
$$
V_{-1}(A_n)^{(q)}=\sum_{j\ge0}L_{A_1}((2j-q)\omega_1)\otimes L_{A_{n-2}}(q\omega_{1})\otimes L_{\C\varpi}(q\zeta).
$$
If $n=3$,
$$
V_{-1}(A_3)^{(q)}=\sum_{j,j'\ge0}L_{A_1}((2j+|q|)\omega_1)\otimes L_{A_1}((2j'+|q|)\omega_1)\otimes L_{\C\varpi}(q\zeta).
$$
In both cases the $q$-eigenspace of $\varpi_{(0)}$ decomposes with infinitely many factors.

\subsection{Proof of Theorem \ref{infinidec} (3)}\label{ex53} We now discuss the embedding of type $A_{1}\times \C\varpi$ in $A_{2}$, conformal level  $k=-\frac{3}{2}$.

We consider the lattice vertex algebra $V_L$
associated to the lattice $L =\ganz\a+\ganz\be+\ganz\d$ such that 
$$
\langle\a,\a\rangle= -\langle\be,\be\rangle =\langle\d,\d\rangle= 1
$$
(other products of basis vectors are zero).
Let $F_{-1}$ be the lattice vertex algebra associated to the lattice $\ganz \phi$ with  $\langle\phi, \phi\rangle = -1$. The operator $\phi_{(0)}$ acts semisimply on $F_{-1}$ and it defines a $\ganz$- gradation
$$
F_{-1}=\oplus_{\ell\in\ganz}F_{-1}^\ell,\quad {\phi_{(0)}}_{|F_{-1}^\ell} =-\ell Id.
$$

In \cite{A-2014} the following facts are shown:
\begin{enumerate}
\item The simple $N = 4$ superconformal vertex algebra $V = L^{N=4}_c$ is realized as a subalgebra of $V_L$.
\item The operator $\d_{(0)}$ acts semisimply on $V$ and  defines a $\ganz$-gradation
$$
V=\oplus_{\ell\in\ganz}V^\ell,\quad {\d_{(0)}}_{|V^\ell} =\ell Id.
$$
\item $V_{-3/2}(sl(2))$ embeds in $V^0$. This turns $V_L$ into a $\widehat{sl(2)}$-module and $e^{t\d}$ is a singular vector for this action for all $t\in\ganz$.
\item
Let $Q=e^{\a+\be-2\d}_{(0)}$. Then $Q$ commutes with the action of $\widehat{sl(2)}$.
\item $V$ is  the maximal $sl(2)$-integrable
part of the Clifford-Weyl vertex algebra $M\otimes F \subset V_L$.
\item 
The subalgebra
$\oplus_{\ell\in\ganz}V^\ell\otimes F_{-1}^\ell$ of $V\otimes F_{-1}$ is isomorphic to the simple affine vertex algebra $V_{-3/2}(sl(3))$.
\end{enumerate}

As a consequence of (1)--(6) we have
\begin{theorem}
For each $q\in\ganz$ there are infinitely many $\widehat{sl(2)}$-singular vectors in $V_{-3/2}(sl(3))^{(q)}$.
In particular $V_{-3/2}(sl(3))^{(q)}$ does not decompose finitely as a $\widehat{sl(2)}$-module.
\end{theorem}
\begin{proof}
The embedding $\widetilde V_{-3/2}(sl(2)\oplus\C\varpi)\subset V_{-3/2}(sl(3))$ corresponds to the pair of embeddings  
$$V_{-3/2}(sl(2))\subset V^0\otimes \C\subset V^0\otimes F_{-1}^0,\quad \C\varpi\subset\C\otimes F_{-1}^0,$$ the rightmost one  mapping $\varpi$ to $\vac\otimes\phi_{(0)}\vac$.

Since $e^{t\d}$ is singular for $\widehat{sl(2)}$ and $Q$ commutes with the action of $\widehat{sl(2)}$, we have that $Q^me^{t\d}$ is singular for all $t\in \ganz$, $m\in\nat$. By \cite{AM1}, $Q^je^{t\d}\ne 0$ iff $0\le j\le t$. It is easy to check that $Q^je^{t\d}$ is $sl(2)$-integral, so $Q^je^{t\d}\in V$. Note also that $Q^je^{t\d}\in V^{t-2j}$, hence $Q^je^{t\d}\otimes e^{(t-2j)\phi}\in V^{t-2j}\otimes F_{-1}^{t-2j}\subset V_k(sl(3))$ and it is clearly a $\widehat{sl(2)\oplus\C\varpi}$-singular vector.

It follows that, if $\ell\ge 0$,   the vectors $v_{\ell,j}=Q^je^{(\ell+2j)\d}\otimes e^{\ell\phi}$ are nonzero $\widehat{\g^0}$-singular vectors for all $j\ge 0$. Since the $sl(2)\oplus \C\varpi$-weight of $v_{\ell,j}$ is $(\ell+2j)\omega_1+\ell\zeta$, we see that the set  $\{v_{\ell,j}\mid j\ge 0\}$ provides an infinite family of linearly independent singular vectors in the $\ell$-eigenspace of $\varpi_{(0)}$.

If $\ell\le 0$, then  the vectors $v_{\ell,j}=Q^{j-\ell}e^{(-\ell+2j)\d} \otimes e^{\ell\phi}$ are nonzero $\widehat{\g^0}$-singular vectors for all $j\ge 0$. Since the $sl(2)\oplus \C\varpi$-weight of $v_{\ell,j}$ is $(-\ell+2j)\omega_1+\ell\zeta$, we see that the set  $\{v_{\ell,j}\mid j\ge 0\}$ provides an infinite family of linearly independent singular vectors in the $\ell$-eigenspace of $\varpi_{(0)}$.
\end{proof}

\subsection{ Proof of Theorem \ref{infinidec} (4)}

Now we discuss embedding of $A_1 \times {\mathbb C}  \varpi  $ in $B_2 =C_2$ at  conformal level $k =-1/2$. 
 We know that $V_{-1/2} (B_2)$ is an even subalgebra of the Weyl vertex algebra $F$ generated by $a_i ^{\pm}$, $i=1,2$, and Kac-Wakimoto free field realization gives that $V_{-1} (sl_2) $ is realized as the subalgebra of $F$ generated by
\begin{equation}\label{efh} e= (a_1^ + )_{-1} a_2 ^{-} , \ f= (a_2 ^ +) _{-1} a_1 ^{-} ,  \ h=- (a_1^ +) _{-1} a_1 ^-  + ( a_2^ +) _{-1} a_2 ^-.
\end{equation}
One checks that  $\varpi =  (a_1^ +) _{-1} a_1 ^-  + ( a_2^ +) _{-1} a_2 ^-.$ By using again \cite[Remark 3.3]{Kw}  we have the following decomposition
$$V_{-1/2} (B_2)= \bigoplus_{q \in \Z}  F^ {2 q},$$
and 
\begin{equation}\label{f2q}F^{2q} = \bigoplus_{j =0}^{\infty} L_{A_1}(( 2j + 2 \vert q\vert )\omega_1) \otimes L_{\C \varpi} (2 q \varpi).\end{equation}
\begin{rem}
Denote by $v_{j,q}$ the singular vector  in $V_{-1/2} (B_2)$ such that
$$V_{-1}(gl_2)\cdot v_{j,q} =L_{A_1}(( 2j + 2  \vert q\vert )\omega_1) \otimes L_{\C \varpi} (2 q \varpi)$$
(cf. \eqref{f2q}).

 We shall  now provide a sketchy derivation of explicit  formulas for these  singular vectors.
Consider the lattice vertex algebra $V_L$ associated to the lattice
$L= {\Z} {\alpha_1} + {\Z} {\alpha_2} +{\Z}{\beta_1} + {\Z}{\beta_2}$ such that
$$ \langle \alpha_i , \alpha_j \rangle = -\langle \beta_i, \beta_j\rangle = \delta_{i,j}, \quad \langle \alpha_i, \beta_j \rangle = 0, \qquad 
i, j \in \{1, 2\}. $$
Then the Weyl vertex algebra $F$ introduced above is isomorphic to the subalgebra of $V_L$ generated by 
$$ a_1 ^+  = e^{\alpha_1 + \beta_1}, \  a_1 ^- = - \alpha_1 (-1)  e^{-\alpha_1 - \beta_1},  \ a_1 ^+ = \alpha_2 (-1)  e^{-\alpha_1 - \beta_1}, \  a_2 ^-  = e^{\alpha_2 + \beta_2}. $$
The simple affine vertex algebra $V_{-1}(sl_2)$ is realized as a subalgebra of $F$ generated by $e,f,h$ given by \eqref{efh}. Note that 
$h= - (\beta_1 + \beta_2)$ and that  $\varpi  =\beta_1 -  \beta_2.$
We have  the following screening operators $$Q^+ = e ^{\alpha_1 - \alpha_2} _0 = \mbox{Res}_z Y(e ^{\alpha_1 - \alpha_2}, z),\quad  Q^- = e ^{\alpha_2 - \alpha_1} _0 = \mbox{Res}_z Y(e ^{\alpha_2 - \alpha_1}, z). $$ 

Set $\delta = 2 \alpha_2 + \beta_1 + \beta_2$, $\varphi = \beta_2-\beta_1$. Then the singular vectors are given by the following formulas
$$ v_{j,n} =(Q^+ ) ^ j  e ^{ (j+ n )  \delta  + n \varphi }, \quad v_{j,-n} =(Q^+ ) ^ {2n + j }  e ^{ (j+ n )  \delta  - n \varphi } \qquad (n \ge 0). $$
The proof that $v_{j,n}$ are elements of  $V_{-1/2} (B_2) \subset F$ uses description of the Weyl vertex algebra as kernel of certain screening operators, and it is omitted.

\end{rem}

\section{Explicit decompositions in the semisimple case}\label{6}
As explained in the proof of Proposition \ref{cw}, we have a complete description of the singular vectors occurring in finite decompositions, thus we can compute explicitly the summands of the decomposition. The list of all decompositions is given below. For $k=1$ all
decompositions give simple current extensions  (see \cite{KMP} for definitions).
In the  non-semisimple cases with finite decomposition we deal with, the explicit decomposition is given by formula \eqref{efd}.

\vskip15pt
\subsection{Type $D$ (level $1$)}

\begin{center}
\begin{tabular}{c|c}
$\g^0$&Decomposition\\
\hline\noalign{\smallskip}
$D_h\times D_{n-h}$, $h\ge 4$, $n-h\ge 4$&$(\dot\L_0,\ddot\L_0)\oplus(\dot\L_1,\ddot\L_1)$\\
\hline\noalign{\smallskip}
$A_1\times A_{1}\times D_{n-2}$, $n\ge 6$&$(\dot\L_0,\ddot\L_0,\dddot\L_0)\oplus(\dot\L_1,\ddot\L_1,\dddot\L_1)$\\
\hline\noalign{\smallskip}
$A_1 \times A_1\times A_{3}$, $n=5$&$(\dot\L_0,\ddot\L_0,\dddot\L_0)\oplus(\dot\L_1,\ddot\L_1,\dddot\L_2)$\\
\hline\noalign{\smallskip}
$A_1 \times A_1\times A_{1}\times A_1$, $n=4$&$(\dot\L_0,\ddot\L_0,\dddot\L_0,\ddddot\L_0)\oplus(\dot\L_1,\ddot\L_1,\dddot\L_1,\ddddot\L_1)$\\
\hline\noalign{\smallskip}
$A_3\times D_{n-3}$, $n\ge 7$&$(\dot\L_0,\ddot\L_0)\oplus(\dot\L_2,\ddot\L_1)$\\
\hline\noalign{\smallskip}
$A_3\times A_{3}$, $n= 6$&$(\dot\L_0,\ddot\L_0)\oplus(\dot\L_2,\ddot\L_2)$\\
\end{tabular}
\end{center}
\vskip15pt

\subsection{Type $B$}

\begin{center}
\begin{tabular}{c|c|c}
$\g^0$&level&Decomposition\\
\noalign{\smallskip}\hline\noalign{\smallskip}
$D_h\times B_{n-h}$, $h\ge 4$, $n-h\ge 3$&1&$(\dot\L_0,\ddot\L_0)\oplus(\dot\L_1,\ddot\L_1)$\\
\noalign{\smallskip}\hline\noalign{\smallskip}
$A_3\times B_{n-3}$, $n\ge 6$&1&$(\dot\L_0,\ddot\L_0)\oplus(\dot\L_2,\ddot\L_1)$\\
\noalign{\smallskip}\hline\noalign{\smallskip}
$A_1 \times A_1\times B_{n-2}$, $n\ge 5$&1&$(\dot\L_0,\ddot\L_0,\dddot\L_0)\oplus(\dot\L_1,\ddot\L_1,\dddot\L_1)$\\
\noalign{\smallskip}\hline\noalign{\smallskip}
$D_{n-2}\times C_{2}$, $n\ge 6$&1&$(\dot\L_0,\ddot\L_0)\oplus(\dot\L_1,\ddot\L_2)$\\
\noalign{\smallskip}\hline\noalign{\smallskip}
$A_1 \times A_1\times C_{2}$, $n= 4$&1&$(\dot\L_0,\ddot\L_0,\dddot\L_0)\oplus(\dot\L_1,\ddot\L_1,\dddot\L_2)$\\
\noalign{\smallskip}\hline\noalign{\smallskip}
$A_3\times C_{2}$, $n=5$&1&$(\dot\L_0,\ddot\L_0)\oplus(\dot\L_2,\ddot\L_2)$\\
\noalign{\smallskip}\hline\noalign{\smallskip}
$D_{n-1}\times A_{1}$, $n\ge 5$&1&$(\dot\L_0,2\ddot\L_0)\oplus(\dot\L_1,2\ddot\L_1)$\\
\noalign{\smallskip}\hline\noalign{\smallskip}
$A_1 \times A_1\times A_{1}$, $n= 3$&1&$(\dot\L_0,\ddot\L_0.2\dddot\L_0)\oplus(\dot\L_1,\ddot\L_1,2\dddot\L_1)$\\
\noalign{\smallskip}\hline\noalign{\smallskip}
$A_3\times A_{1}$, $n=4$&1&$(\dot\L_0,2\ddot\L_0)\oplus(\dot\L_2,2\ddot\L_1)$\\
\noalign{\smallskip}\hline\noalign{\smallskip}
$D_{n}$, $n\ge 4$&$\frac{3}{2}-n$&$((\frac{3}{2}-n)\dot\L_0)\oplus((\frac{1}{2}-n)\dot\L_0+\dot\L_1)$\\
\noalign{\smallskip}\hline\noalign{\smallskip}
$A_3$, $n=3$&$-\frac{3}{2}$&$(-\frac{3}{2}\dot\L_0)\oplus(-\frac{5}{2}\dot\L_0+\dot\L_1)$\\
\end{tabular}
\end{center}
\vskip15pt
\subsection{Type $C$ (level $-\tfrac{1}{2}$)}

\begin{center}
\begin{tabular}{c|c}
$\g^0$&Decomposition\\
\hline\noalign{\smallskip}
$C_h\times C_{n-h}$, $h\ge 2$, $n-h\ge 2$&$(-\frac{1}{2}\dot\L_0,-\frac{1}{2}\ddot\L_0)\oplus(-\frac{3}{2}\dot\L_0+\dot\L_1,-\frac{3}{2}\ddot\L_0+\ddot\L_1)$\\
\noalign{\smallskip}\hline\noalign{\smallskip}
$A_1\times C_{n-1}$, $n\ge 3$&$(-\frac{1}{2}\dot\L_0,-\frac{1}{2}\ddot\L_0)\oplus(-\frac{3}{2}\dot\L_0+\dot\L_1,-\frac{3}{2}\ddot\L_0+\ddot\L_1)$\\
\noalign{\smallskip}\hline\noalign{\smallskip}
$A_1 \times A_1$, $n=2$&$(-\frac{1}{2}\dot\L_0,-\frac{1}{2}\ddot\L_0)\oplus(-\frac{3}{2}\dot\L_0+\dot\L_1,-\frac{3}{2}\ddot\L_0+\ddot\L_1)$
\end{tabular}
\end{center}
\vskip15pt
\subsection{Type $E_6$ (level $1$)}

\begin{center}
\begin{tabular}{c|c}
$\g^0$&Decomposition\\
\hline\noalign{\smallskip}
$A_{1}\times A_5$&$(\dot\L_0,\ddot\L_0)+(\dot\L_1,\ddot\L_3)$\\
\hline\noalign{\smallskip}
$A_{2}\times A_2 \times A_2$&$(\dot\L_0,\ddot\L_0,\dddot\L_0)+(\dot\L_1,\ddot\L_1,\dddot\L_1)+(\dot\L_2,\ddot\L_2,\dddot\L_2)$\\
\end{tabular}
\end{center}
\subsection{Type $E_7$ (level $1$)}

\begin{center}
\begin{tabular}{c|c}
$\g^0$&Decomposition\\
\hline\noalign{\smallskip}
$A_{1}\times D_6$&$(\dot\L_0,\ddot\L_0)+(\dot\L_1,\ddot\L_6)$\\
\hline\noalign{\smallskip}
$A_{2}\times A_5$&$(\dot\L_0,\ddot\L_0)+(\dot\L_1,\ddot\L_4)+(\dot\L_2,\ddot\L_2)$\\\hline
$A_7$&$\dot\L_0+\dot\L_4$\\
\end{tabular}
\end{center}
\vskip15pt
\subsection{Type $E_8$ (level $1$)}

\begin{center}
\begin{tabular}{c|c}
$\g^0$&Decomposition\\
\hline\noalign{\smallskip}
$A_{1}\times E_7$&$(\dot\L_0,\ddot\L_0)+(\dot\L_1,\ddot\L_7)$\\
\hline\noalign{\smallskip}
$A_{2}\times E_6$&$(\dot\L_0,\ddot\L_0)+(\dot\L_2,\ddot\L_1)+(\dot\L_1,\ddot\L_6)$\\\hline
$A_{4}\times A_4$&$(\dot\L_0,\ddot\L_0)+(\dot\L_4,\ddot\L_2)+(\dot\L_1,\ddot\L_3)+(\dot\L_2,\ddot\L_4)+(\dot\L_3,\ddot\L_1)$\\\hline
$D_8$&$\dot\L_0+\dot\L_7$\\\hline
$A_8$&$\dot\L_0+\dot\L_6+\dot\L_3$
\end{tabular}
\end{center}
\vskip15pt
\subsection{Type $F_4$}

\begin{center}
\begin{tabular}{c|c|c}
$\g^0$&level&Decomposition\\
\hline\noalign{\smallskip}
$A_{1}\times C_3$&1&$(\dot\L_0,\ddot\L_0)+(\dot\L_1,\ddot\L_3)$\\
\hline\noalign{\smallskip}
$A_{2}\times A_2$&1&$(\dot\L_0,2\ddot\L_0)+(\dot\L_2,2\ddot\L_1)+(\dot\L_1,2\ddot\L_2)$\\\hline
$B_4$&$-\frac{5}{2}$&$(-\frac{5}{2}\dot\L_0)+(-\frac{7}{2}\dot\L_0+\dot\L_4).$
\end{tabular}
\end{center}
\vskip15pt
\subsection{Type $G_2$}

\begin{center}
\begin{tabular}{c|c|c}
$\g^0$&level&Decomposition\\
\hline\noalign{\smallskip}
$A_{1}\times A_1$&1&$(\dot\L_0,\ddot\L_0)+(\dot\L_1,3\ddot\L_1)$\\
\hline\noalign{\smallskip}
$A_2$&$-\frac{5}{3}$&$(-\frac{5}{3}\dot\L_0)+(-\frac{8}{3}\dot\L_0+\dot\L_1)+(-\frac{8}{3}\dot\L_0+\dot\L_2).$
\end{tabular}
\end{center}

  \vskip10pt
  \footnotesize{
  \noindent{\bf D.A.}:  Department of Mathematics, University of Zagreb, Bijeni\v{c}ka 30, 10 000 Zagreb, Croatia;
{\tt adamovic@math.hr}
  
\noindent{\bf V.K.}: Department of Mathematics, MIT, 77
Mass. Ave, Cambridge, MA 02139;\newline
{\tt kac@math.mit.edu}

\noindent{\bf P.MF.}: Politecnico di Milano, Polo regionale di Como,
Via Valleggio 11, 22100 Como,
Italy; {\tt pierluigi.moseneder@polimi.it}

\noindent{\bf P.P.}: Dipartimento di Matematica, Sapienza Universit\`a di Roma, P.le A. Moro 2,
00185, Roma, Italy; {\tt papi@mat.uniroma1.it}

\noindent{\bf O.P.}:  Department of Mathematics, University of Zagreb, Bijeni\v{c}ka 30, 10 000 Zagreb, Croatia;
{\tt perse@math.hr}
}

\end{document}